\documentclass[a4paper]{amsart}
\usepackage[hypertex]{hyperref}
\usepackage{amsmath,amsthm,amssymb, amscd, amsfonts}
\usepackage[all]{xy}
\usepackage{leftidx}
\usepackage{bbold}
\usepackage[latin1]{inputenc}
\usepackage{enumerate}

\theoremstyle{plain}

\newtheorem{theorem}{Theorem}[section]
\newtheorem{corollary}[theorem]{Corollary}

\newtheorem{proposition}[theorem]{Proposition}
\newtheorem{lemma}[theorem]{Lemma}

\theoremstyle{definition}

\newtheorem{example}[theorem]{Example}

\theoremstyle{remark}
\newtheorem{remark}[theorem]{Remark}

\numberwithin{equation}{section}\theoremstyle{plain}

\newcommand{\dar}[2]{\ar@<2pt>[r]^-{#1}\ar@<-2pt>[r]_-{#2}}

\newcommand{\prettydef}[1]{\left\{\begin{array}{ccl}#1\end{array}\right.}

\newcommand{\ti}{\mbox{-}}

\newcommand{\lcoev}{\mathrm{coev}}

\newcommand{\Ob}{\mathrm{Ob}}

\newcommand{\bra}[1]{\langle #1 \rangle}
\newcommand{\iso}{\stackrel{\sim}{\longrightarrow}}

\newcommand{\eps}{\varepsilon}

\newcommand{\ldual}[1]{\leftidx{^\vee}{\!#1}{}}
\newcommand{\rdual}[1]{{#1}^\vee}

\newcommand{\kk}{\Bbbk}
\newcommand{\kt}{$\Bbbk$\nobreakdash-\hspace{0pt}}

\newcommand{\OO}{\mathcal{O}}

\newcommand{\un}{\mathbb{1}}

\newcommand{\car}{\mathrm{char}}

\newcommand{\Aa}{{\mathcal A}}
\newcommand{\B}{{\mathcal B}}
\newcommand{\C}{{\mathcal C}}
\newcommand{\D}{{\mathcal D}}
\newcommand{\LL}{{\mathcal L}}
\newcommand{\T}{{\mathcal T}}
\newcommand{\Z}{{\mathcal Z}}
\newcommand{\Kh}{{K}}

\newcommand{\M}{\mathcal{M}}
\newcommand{\N}{\mathcal{N}}

\newcommand{\TY}{\mathcal{TY}}

\newcommand{\Zed}{{\mathbb Z}}

\newcommand{\E}{{\mathcal E}}
\newcommand{\U}{{\mathcal U}}
\newcommand{\A}{\mathbb A}
\newcommand{\Spec}{\operatorname{Spec}}
\newcommand{\Pic}{\operatorname{Pic}}
\newcommand{\pt}{\mathrm{pt}}
\newcommand{\ab}{\mathfrak{ab}}
\newcommand{\KER}{\mathfrak{Ker}}
\newcommand\CoRep[1]{\operatorname{comod}\!\mbox{-} #1}
\newcommand\Rep[1]{\operatorname{mod}\!\mbox{-} #1}
\newcommand\rep{\operatorname{rep}}
\newcommand\coend{\operatorname{coend}}

\newcommand\Irr{\operatorname{Irr}}
\newcommand\FPdim{\operatorname{FPdim}}
\newcommand\FPind{\operatorname{FPind}}
\newcommand\vect{\operatorname{vect}}
\newcommand\id{\operatorname{id}}

\newcommand\End{\operatorname{End}}

\newcommand\Hom{\operatorname{Hom}}

\newcommand\dys{\operatorname{dys}}

\newcommand{\EndFun}{\underline{\operatorname{End}}}

\newcommand{\toto}{\longrightarrow}

\usepackage[dvipsnames]{xcolor}

\newcommand{\eab}{\normalcolor{}}

\begin{document}
\title[ Central exact sequences of tensor categories]{ Central exact sequences
of tensor categories, equivariantization and applications}
\author{Alain Brugui\`{e}res}
\author{Sonia Natale}
\address{Alain Brugui\`{e}res: D\' epartement de Math\' ematiques.
\; Universit\' e Montpellier II. Place Eug\` ene Bataillon. 34 095
Montpellier, France} \email{bruguier@math.univ-montp2.fr
\newline \indent \emph{URL:}\/ http://www.math.univ-montp2.fr/~bruguieres/}
\address{Sonia Natale: Facultad de Matem\'atica, Astronom\'\i a y F\'\i sica.
Universidad Nacional de C\'ordoba. CIEM -- CONICET. Ciudad
Universitaria. (5000) C\'ordoba, Argentina}
\email{natale@famaf.unc.edu.ar
\newline \indent \emph{URL:}\/ http://www.famaf.unc.edu.ar/$\sim$natale}

\thanks{The work of the second author was partially supported by  CONICET, SeCYT--UNC and
Mathamsud project NOCOMALRET}


\date{\today.}

\begin{abstract} We define equivariantization of tensor categories under tensor group scheme actions and give  necessary and sufficient conditions for an exact sequence of tensor categories to be
an equivariantization under a finite group or finite group scheme
action. We introduce the notion of central exact sequence of
tensor categories and use it in order to present an alternative
formulation of some known characterizations of equivariantizations
for fusion categories, and to extend these characterizations to
equivariantizations of finite tensor categories under  finite
group scheme actions. In particular, we obtain a simple
characterization of equivariantizations under actions of finite
abelian groups. As an application, we show that if $\C$ is a
fusion category and $F: \C \to \D$ is a dominant tensor functor of
Frobenius-Perron index $p$, then $F$ is an equivariantization if
$p=2$, or if $\C$ is weakly integral and $p$ is the smallest prime
factor of $\FPdim \C$.
%
\end{abstract}

\maketitle

\section{Introduction}

In this paper we pursue the study of
exact sequences of tensor categories initiated
in \cite{tensor-exact}.
Exact sequences of tensor categories generalize (strict) exact sequences of Hopf algebras, due do Schneider, and in particular, exact sequences of groups.


By a tensor category over a field $\kk$, we mean a monoidal rigid category $(\C,\otimes,\un)$ endowed with a \kt linear abelian
structure such that
\begin{itemize}
\item $\Hom$ spaces are finite dimensional and all objects have finite length,
\item the tensor product $\otimes$ is \kt linear in each variable and the unit object $\un$ is scalar, that is,  $\End(\un) = \kk$.
\end{itemize}
A tensor category is \emph{finite} if it is \kt linearly equivalent to the category of finite dimensional right modules over a finite dimensional \kt algebra.

We will mostly work with finite tensor categories, with special attention to \emph{fusion categories}. A fusion category is a split semisimple  finite tensor category (\emph{split} semisimple means semisimple with scalar simple objects).

A \emph{tensor functor} is a strong monoidal, \kt linear exact
functor between tensor categories over $\kk$; it is  faithful. A
tensor functor $F:\C \to \D$ is \emph{dominant}\footnote{Dominant
functors  between finite tensor categories  are called
\emph{surjective} in \cite{ENO}.} if any object $Y$ of $\D$ is a
subobject of $F(X)$ for some object $X$ of $\C$.

A tensor functor $F: \C \to D$ is \emph{normal} if any object $X$ of $\C$ admits a subobject $X' \subset X$ such that $F(X')$ is the largest
trivial subobject of $F(X)$. An object is \emph{trivial} if it is isomorphic to $\un^n$ for some natural integer $n$.

We denote by $\KER_F \subset \C$ the full tensor subcategory of objects $X$ of $\C$ such that $F(X)$ is trivial.

%

%

Let $\C'$, $\C$, $\C''$ be tensor categories over $\kk$. A
sequence of tensor functors

\begin{equation}\label{suite}\C' \overset{i}\toto \C \overset{F}\toto \C''
\end{equation}
is called an \emph{exact sequence of tensor categories} if
\begin{itemize}
\item $F$
is dominant and normal,
\item $i$ is a full embedding whose essential image is  $\KER_F$.
\end{itemize}

An exact sequence of finite tensor categories $\C' \toto \C
\overset{F}\toto \C''$ is \emph{perfect} if the left (or
equivalently, the right) adjoint of $F$ is exact. Such is
always the case if $\C''$ is a fusion category.

\subsection*{The monadic approach} Let $F: \C \to \C''$ be a tensor functor between finite tensor categories. Then $F$ is monadic, that is, it admits a left adjoint $L$. The endofunctor $T = FL$ of $\C''$ is a \kt linear Hopf monad on $\C''$, and $\C$ is tensor equivalent to the category ${\C''}^T$ of $T$-modules in $\C''$. Moreover, $F$ is dominant if and only if $T$ is faithful, and $F$ is normal if and only if $T(\un)$ is trivial, in which case $T$ is said to be normal.

Via this construction, exact sequences of finite tensor categories
$$\C' \overset{i}\toto \C \overset{F}\toto \C''$$ are classified by \kt linear right exact faithful normal Hopf monads on $\C''$ \cite[Theorem 5.8]{tensor-exact}.

\subsection*{Examples}
Any (strictly) exact sequence of Hopf algebras  $H' \to H \to H''$
over a field $\kk$ in the sense of Schneider~\cite{schneider}
gives rise to an exact sequence of tensor categories of finite
dimensional comodules:
$$(\E_S)\qquad\CoRep{H'} \to \CoRep{H} \to \CoRep{H''},$$
and if $H$ is finite-dimensional we also have an exact sequence of tensor categories of finite dimensional modules:
$$\Rep{H''} \to \Rep{H} \to \Rep{H'}.$$

Equivariantization is another source of examples. Let $G$ be a
finite group acting on a tensor category $\D$ by tensor
autoequivalences. Then the \emph{equivariantization} $\D^G$ is a
tensor  category and the forgetful functor $\D^G \to \D$ gives
rise to a  (perfect) exact sequence of tensor categories
\begin{equation}\label{suite-equiv}\rep G \to \D^G \to
\D,\end{equation} see \cite[Section 5.3]{tensor-exact}. This is
extended in Section~\ref{sect-group-scheme} to the case where $G$
is a finite group  scheme. If $\D$ is a fusion category,
$\kk$ is algebraically closed and the order of $G$ is not a
multiple of $\car(\kk)$, then \eqref{suite-equiv} is an exact
sequence of fusion categories.

A tensor functor $F: \C \to \D$ is an \emph{equivariantization} if there is an action of a finite group scheme $G$ on $\D$ and a tensor equivalence $\C \simeq \D^G$ such that the triangle of tensor functors $$\xymatrix@!0 @R=6mm @C=10mm{\C \ar[rr]^{\simeq} \ar[rd]& & \D^G\ar[ld]\\ &\D& }$$
commutes up to a \kt linear monoidal isomorphism.

An exact sequence of tensor categories $\C' \overset{i}\toto \C
\overset{F}\toto \C''$ is called an \emph{equivariantization exact
sequence}  if  it is equivalent to an exact sequence
defined by an equivariantization, or equivalently, if $F$ is an
equivariantization.


A \emph{braided exact sequence} is an exact sequence of tensor
categories where all categories and functors are braided. If $\C'
\to \C \to \C''$ is a braided exact sequence, then $\C'$ is a
subcategory of the category $\T \subset \C$ of transparent objects
of $\C$ (see \cite{bruguieres}). We say that $\C' \to \C \to \C''$
is a \emph{modularization exact sequence} if $\C''$ is modular,
that is, if all transparent objects of $\C''$ are trivial. In that
case $\C' = \T$. Examples of  modularization  exact
sequences of fusion categories arise through the modularization
procedures introduced in \cite{bruguieres, mueger-mod}.

\subsection*{Equivariantization criteria}
Generalizing a result of \cite{tensor-exact}, we show  that an
exact sequence of finite tensor categories $\C' \to \C \to \C''$
is an equivariantization exact sequence if and only if the
associated normal Hopf monad is exact and cocommutative in the
sense of \cite{tensor-exact}. If  $\C'$ is finite and $\kk$ is an
algebraically closed field such that $\car(\kk)$ does not divide
$\dim\C'$, then the corresponding group scheme is discrete so that
we have an equivariantization in the usual sense.

In particular, any  perfect braided  exact sequence  of finite
tensor categories is an equivariantization exact sequence.

On the other hand, \cite[Proposition 2.10]{ENO2}(i) affirms that a fusion category $\C$ is an equivariantization under the action of a finite group $G$ if there is a full braided embedding $j$ of the category $\rep G$ into  into the Drinfeld center $\Z(\C)$ of $\C$, such that
$Uj: \rep G  \to \C$ is full, where $U$ denotes the forgetful functor $\Z(\C) \to \C$.
See also \cite[Theorem 4.18]{DGNOI}.

In order to unify those two result, we introduce the notion of central exact sequence of tensor categories.
An exact sequence of finite tensor categories
$$\C' \overset{i}\toto \C \overset{F}\toto \C''$$
is \emph{central} if, denoting by $(A,\sigma)$ its central
commutative algebra, the tensor functor $i: \C' \to \C$  lifts to
a tensor functor $\tilde{i}: \C' \to \Z(\C)$ such that
$\tilde{i} (A)=(A,\sigma)$. Such a lift, if it exists, is
essentially unique.

We show that an exact sequence of finite tensor categories  is
 central if and only if its normal Hopf monad is
cocommutative.
%
We give two proofs of this result.

The first one works in the fusion case. In that situation our
characterization is a reformulation in terms of exact sequences of
the characterization of equivariantizations given in
\cite[Proposition 2.10]{ENO2}, \cite[Theorem 4.18]{DGNOI}. However
our proof is organized differently, and boils down to showing that
a central exact sequence of fusion categories is `dominated' in a
canonical way by a modularization exact sequence, thus:
$$
\xymatrix {
\C' \ar[r] \ar[d]_{=} & C_{\Z(\C)}(\C') \ar[r]\ar[d] &  \Z(\C'') \ar[d]\\
\C' \ar[r] & \C \ar[r] & \C'',
}
$$
where $C_{\Z(\C)}(\C')$ denotes the centralizer of $\C'$ viewed as a fusion subcategory of $\Z(\C)$,
and that an exact sequence dominated by an equivariantization exact sequence is itself an equivariantization exact sequence.

The second one works for finite tensor (not necessarily
semisimple) categories and actions of finite group schemes, and
relies on the construction of the double of a Hopf monad in
\cite{BV}. It is based on the existence of a commutative diagram
of tensor categories
$$
\xymatrix {
\C' \ar[r] \ar[d]_{=} & \Z(\C) \ar[r]\ar[d] &  \Z_F(\C'') \ar[d]\\
\C' \ar[r] & \C \ar[r]^F & \C'',
}
$$
whose first line is an exact sequence of tensor categories, $\Z_F(\C'')$ denoting the center of $\C''$ relative to the functor $F$.

\subsection*{Application}

The \emph{Frobenius-Perron index} of
a dominant tensor functor $F: \C \to \D$ between fusion categories is defined in \cite{tensor-exact} to be
the ratio
$$\FPind(F) = \FPind(\C: \D) = \frac{\FPdim \C}{\FPdim \D}.$$
It is an algebraic integer by \cite[Corollary 8.11]{ENO}.
According to \cite[Proposition 4.13]{tensor-exact}, a dominant tensor functor $F$ of Frobenius-Perron index $2$ is normal.
In this paper we prove the following refinement of this result:

\newtheorem*{theo1}{Theorem 6.1}
\newtheorem*{theo2}{Theorem 6.2}

\begin{theo1}\label{index-2} Let $F: \C \to \D$ be a dominant tensor functor between fusion categories over a field
of characteristic $0$. If $\FPind(\C: \D) = 2$, then $F$ is an equivariantization.
\end{theo1}

This  generalizes the fact that that subgroups of index $2$ are
normal. An analogue in the context of finite dimensional
semisimple Hopf algebras was proved in \cite[Proposition
2]{kob-mas}, \cite[Corollary 1.4.3]{ssld}. We also generalize the
fact that subgroups of a finite group whose index is the smallest
prime factor of the order of the larger group are normal. Recall
that a fusion category is \emph{weakly integral} if its
Frobenius-Perron dimension is a natural integer.

\begin{theo2}\label{index-p} Let $F: \C \to \D$ be a dominant tensor functor between fusion categories over a field of characteristic $0$. Assume that $\FPdim\, \C$ is a natural integer, and that $\FPind(\C: \D)$ is the smallest prime number dividing $\FPdim\,\C$. Then $F$ is an equivariantization.
\end{theo2}

In particular under the hypotheses of Theorem \ref{index-p} the functor
$F$ is normal.
An analogue in the context of
semisimple Hopf algebras was proved in \cite[Proposition
2]{kob-mas}, \cite[Corollary 1.4.3]{ssld}.
Observe
that Theorem \ref{index-p} gives some positive evidence in favor of the
conjecture that every weakly integral fusion category is weakly
group-theoretical \cite{ENO2}.

Regarding the `dual' situation, namely, when $\C$ is a weakly
integral fusion category and $\D \subseteq \C$ is a full fusion
subcategory such that the quotient $\frac{\FPdim \C}{\FPdim \D}$
is the smallest prime factor of $\FPdim \C$, it may be the case
that $\D$ is \emph{not} normal in $\C$; we give an example of this
where $p=2$ and $\C$ is a Tambara-Yamagami category (see
Proposition \ref{simple-ty}). This actually provides examples of
simple fusion categories of Frobenius-Perron dimension $2q$, where
$q$ is an odd prime number.

\subsection*{Organization of the text}
In Section~\ref{sect-exact} we introduce central exact sequences
of tensor categories. We also discuss dominant tensor functors on
weakly integral fusion categories in terms of induced central
algebras and show that this class of fusion categories is closed
under extensions; see Corollary \ref{wi-closed}. We give a general
criterion for an exact sequence to be an equivariantization exact
sequence in Section~\ref{sect-equi-general}. In order to do so, we
generalize the notion of equivariantization to tensor actions of
group schemes on tensor categories in 
Section~\ref{sect-group-scheme}.  The main results,
Theorems~\ref{equiv-monad-general} and \ref{cocom-central}, assert
that equivariantization exact sequences coincide with central
exact sequences, and also with exact sequences whose Hopf monad is
exact  cocommutative.  It is proved in Section~\ref{proof-main}
using the notion of double of a Hopf monad. We apply this
characterization in Section~\ref{sect-special} to several special
cases: braided exact sequences of tensor categories, exact
sequences of fusion categories, equivariantizations under the
action of abelian groups. Lastly we prove Theorems~\ref{index-2}
and~\ref{index-p} in Section \ref{demo}.


\subsection*{Conventions and notation}
We retain the conventions and notation of \cite{tensor-exact}.

If $\C$ is a monoidal category and $A$ is an algebra in $\C$, we denote by $\C_A$ the category of right $A$\ti modules
in $\C$, and by $F_A: \C \mapsto \C_A$ the free $A$\ti module functor, defined by $X \mapsto X \otimes A$.
If $\C$ is additive, so is $\C_A$. In that case, we say that $A$ is \emph{self-trivializing} if $F_A(A) \simeq F_A(\un)^n$ for some natural integer $n$.

Let  $\C$ be a tensor category over a field $\kk$, and let
${\mathcal X}$ be a object,  or a set of objects of $\C$. We
denote by $\bra{\mathcal X}$ the smallest full replete tensor
subcategory of $\C$ containing ${\mathcal X}$. Its objects are the
subquotients of finite direct sums of tensor products of elements
of ${\mathcal X}$ and their duals. We denote by $\un$ the unit
object of $\C$. The tensor subcategory $\bra{\un}$ is the category
of trivial objects of $\C$ and it is tensor equivalent to
$\vect_\kk$.

Given an object $X$ of $\C$, we denote by $\ldual{X}$ and $\rdual{X}$ the left dual and the right dual of $X$ respectively.
An object $X$ of $\C$ is called \emph{invertible} if there exists an object $Y$ of $\C$ such that $X \otimes Y \simeq \un \simeq Y \otimes X$. In that case $Y\simeq \ldual{X} \simeq \rdual{X}$. Invertible objects of $\C$ are both simple and scalar. We denote by $\Pic(\C)$ the set of isomorphism classes of invertible objects of $\C$; it is a group for the tensor product, called the \emph{Picard group of $\C$}.  We set $\C_{\pt} = \bra{\Pic(\C)} \subset \C$.

Now assume $\C$ is a fusion category. The multiplicity of a simple object $X$ in an object $Y$ of $\C$ is
defined as $m_X(Y) = \dim \Hom_\C(X, Y)$. We have $$Y \simeq \oplus_{X \in
\Irr(\C)}X^{m_X(Y)},$$ where $\Irr(\C)$ denotes the set of simple objects of $\C$ up to isomorphism.
An object $X$ of $\C$ is invertible if and only if its Frobenius-Perron dimension is $1$.

\section{Central exact sequences of tensor categories}\label{sect-exact}


\subsection{Tensor functors}
 Let  $F: \C \to \D$ be a tensor functor between tensor categories over $\kk$.  We  denote by $\KER_F \subset \C$ the full subcategory of $\C$ of objects $c$ such that $F(c)$ is trivial in $\D$. 
The category
$\KER_F$ is a tensor category over $\kk$, and it is endowed with a fibre functor
$$\omega_F \prettydef{\KER_F & \to & \vect_\kk\\ x & \mapsto & \Hom(\un,F(x))}$$
By Tannaka reconstruction, this defines a Hopf algebra
$$ H = \coend (\omega_F)  = \int^{x \in \KER_F}  \rdual{\omega(x)}  \otimes
\omega(x)$$ such that $\KER_F \simeq \CoRep  H$.

The tensor functor $F$ admits a left adjoint $L$ if and only if it admits a right adjoint $R$; if they exist, the  adjoints of $F$ are related by $R(X) = \rdual{L(\ldual{X})}$. If $\C$ is finite, then $F$ admits adjoints.

If the tensor functor $F$ admits adjoints, we say that $F$ is
\emph{perfect} if its left, or equivalently, its right adjoint is
exact. Such is always the case if $\D$ is a fusion category.

The tensor functor $F$ is \emph{dominant} if for any object $d$ of
$\D$, there exists an object $c$ of $\C$ such that $d$ is a
subobject of $F(c)$. It is \emph{normal}  if  for any
object $c$ of $\C$, there exists a subobject $c_0 \subset c$ such
that $F(c_0)$ is the largest trivial subobject of $F(c)$.

If $F$ admits adjoints $L$ and $R$, then $F$ is dominant if and only if $L$, or equivalently $R$, is faithful, and $F$ is normal if and only if $L(\un)$, or equivalently $R(\un)$, is a trivial object of $\C$.

If $F$ is a normal tensor functor, the Hopf algebra $H =
\coend(\omega_F)$ is called the \emph{induced Hopf algebra of
$F$}.

\subsection{Central induced
(co)algebras}\label{central-alg}

Let $\C$ and $\D$ be finite tensor categories over $\kk$ and let
$F: \C \to \D$ be a  dominant  tensor functor. Then $F$
admits a left adjoint $L: \D \to \C$ which is faithful and
comonoidal. Consequently $\hat{C}= L(\un)$ is  a coalgebra
 in $\C$, called \emph{the induced coalgebra of $F$,} with
coproduct
$L_2(\un, \un)$ and unit $L_0$, where $(L_2, L_0)$ denotes the
comonoidal structure of $L$. We have
 $\Hom_{\C}(\hat{C}, \un) \simeq \kk$.
 In addition, $\hat{C}$ is endowed with a canonical half-braiding
 $\hat{\sigma}: \hat{C} \otimes \id_\C \to \id_\C \otimes \hat{C}$, which makes it a cocommutative coalgebra in
 the center $\Z(\C)$ of $\C$;  see \cite{BLV} for details of this construction.
 The cocommutative coalgebra $(\hat{C},\hat{\sigma})$ is called \emph{the induced central coalgebra of $F$.}

 Dually, under the same hypotheses the functor $F$ also admits a right adjoint $R$, related to $L$ by $R(X) = \rdual{L(\ldual{X})}$. The functor $R$
 is faithful and monoidal. As a result, $A= R(\un) = \rdual{\hat{C}}$ is an algebra in $\C$, called \emph{the induced algebra of $F$,} and it is endowed a canonical half-braiding
 $\sigma: A \otimes \id_\C \to \id_\C \otimes A$, making it a commutative algebra in $\Z(\C)$.
 The commutative algebra $(A,\sigma)$, which is the right dual of $(\hat{C},\hat{\sigma})$, is called \emph{the induced central algebra of $F$.}
See \cite[Section 6]{tensor-exact}.

\medbreak The category $\C_A = \C_{(A, \sigma)}$ of right
$A$-modules in $\C$ is an abelian \kt linear  monoidal  category over $\kk$ with tensor
product induced by $\otimes_A$ and the half-braiding $\sigma$, and
the functor $F_{(A, \sigma)}: \C \to \C_A$, $F_{(A, \sigma)} (X) =
X \otimes A$, is strong monoidal and \kt linear.

If $F$ is dominant and perfect (that is, $R$ is faithful exact),
 then by \cite[Proposition  6.1]{tensor-exact} $\C_A$ is a
tensor category, and there is a tensor equivalence $\kappa : \D
\to \C_A$ such that the following diagram of tensor functors
commutes up to tensor isomorphims:
$$\xymatrix@!0 @R=1.0cm @C=1.7cm{\C \ar[r]^F \ar[rd]_{F_A}& \D\ar[d]^\kappa\\
& \C_A
}$$
Note that if $F$ is dominant and $\D$ is a fusion category, then $R$ is exact, so  $\C_A \simeq \D$ is a fusion category and in that case, $A$ is semisimple.

\begin{lemma}\label{normal-algebra} Let $F: \C \to \D$ be a tensor functor between finite tensor categories, with induced algebra $A$. The following assertions are equivalent:
\begin{enumerate}[(i)]
\item $F$ is normal;
\item $A$ belongs to $\KER_F$;
\item $A$ is a self-trivializing algebra, that is, $A\otimes A \simeq A^n$ in $\C_A$  for some integer $n$.
\end{enumerate}
Moreover if these hold, the integer $n$ of assertion (iii) is the dimension of the induced Hopf algebra of $F$.
\end{lemma}

\begin{proof}
 Since $A=R(\un)$, we have (i) $\iff$ (ii).
 If $A$ is in $\KER_F$, then $F_A(A) \simeq \kappa F(A) \simeq \kappa F(\un)^n \simeq F_A(\un)^n$, \emph{i.e.} $A$ is self-trivializing, so  (ii) $\implies$ (iii). Conversely, assume $A$ is self-trivializing, that is, $A \otimes A \simeq A^n$ as right $A$\ti modules.
By adjunction,  $\Hom_\D(F(A),\un)\simeq \Hom_\C(A,R(\un)) = \Hom_\C(A,A)\simeq\Hom_{\C_A}(F_A(A),F_A(\un)) \simeq \Hom_{\C_A}(F_A(\un)^n,F_A(\un)) \simeq \Hom_\C(\un,A)^n \simeq \kk^n$. Thus, there exists an epimorphism $s: F(A) \to \un^n$ in $\D$.
Now $s \otimes F(A)$ is an epimorphism $F(A) \otimes F(A) \to F(A)^n$, and since those two objects are isomorphic and of finite length, $s\otimes F(A)$ is an isomorphism. Thus $\ker(s)\otimes F(A)=0$,  so $\ker(s) = 0$ and $s$ is an isomorphism $F(A) \iso \un^n$. This shows (iii) $\implies$ (ii).

 If the assertions of the lemma hold, then $F(A) = FR(\un)
\simeq F(\rdual{L(\un)}) \simeq \rdual{F(L(\un))}$. Hence
$\ldual{F(A)} \simeq H \otimes \un$, and we get $n=\dim H$.
\end{proof}

\begin{remark} A terminological summary might help: if $F : \C \to \D$ is a tensor functor between finite tensor categories, the induced central algebra $\A=(A,\sigma)$
of $F$ is a commutative algebra in $\Z(\C)$; the induced algebra
$A$ of $F$ is an algebra in $\C$. If $F$ is normal, with induced
Hopf algebra $H= \coend(\omega_F)$, then $F(A) \simeq H
\otimes \un$ and, denoting by $T = FL$ the Hopf monad of $F$, we
have $T(\un) \simeq  \Kh \otimes \un$ with $ \Kh =
H^*$.
\end{remark}

%
%

\subsection{Exact sequences of tensor categories}

In this section we recall some basic facts about exact sequences
of tensor categories which we will use throughout this paper,  see
\cite{tensor-exact} for details. Let $\kk$ be a field.

An \emph{exact sequence of tensor categories} over $\kk$ is a diagram of tensor functors
$$(\E) \qquad\C' \overset{i}\toto \C \overset{F}\toto \C''$$
between tensor categories $\C'$, $\C$, $\C''$ over $\kk$, such
that $F$ is normal  and dominant , $i(\C') \subset \KER_F$
and $i$ induces a tensor equivalence $\C' \to \KER_F$.

The induced Hopf algebra $H = \coend(\omega_F)$  of $F$ is also
called the \emph{induced Hopf algebra of $(\E)$,} and we have an
equivalence of tensor categories $\C'' \simeq \CoRep H$. \eab By
\cite[Proposition 3.15]{tensor-exact}, the induced Hopf algebra of
$(\E)$ is finite dimensional if and only if the tensor functor $F$
has a left adjoint, or equivalently, a right adjoint. In that
case, we say that $(\E)$ is \emph{perfect} if $F$ is perfect, that
is, $R$ (or $F$) is exact.

If $(\E_1) = (\C'_1 \to \C_1 \to \C''_1)$ and $(\E_2) = (\C'_2 \to
\C_2 \to \C''_2)$ are  two exact sequences of tensor categories
over $\kk$, a \emph{morphism of exact sequences of tensor
categories from  $(\E_1)$ to $(\E_2)$} is a diagram of
tensor functors:
$$\xymatrix{
\C_1' \ar[r]\ar[d] &\C_1 \ar[r]\ar[d] &\C_1''\ar[d] \\
\C_2' \ar[r] &\C_2 \ar[r] & \C_2'' }$$ which commutes up
to tensor isomorphisms. Such a morphism induces a morphism of Hopf
algebras $w :  H_1 \to H_2$, where  $H_1$ and $H_2$
 denote the induced Hopf coalgebras of  $(\E_1)$ and of
$(\E_2)$, respectively.

A morphism of exact sequences of tensor categories is  an
\emph{equivalence of exact sequences of tensor categories} if the
vertical arrows are equivalences.


%


\begin{lemma}\label{exact-dominant} Consider a morphism of exact sequences of finite tensor categories
$$\xymatrix{
(\E_1) \ar[d] &\C_1' \ar[r]^{i_1} \ar[d]_{U'} & \C_1 \ar[r]^{F_1}\ar[d]^{U} &  \C_1'' \ar[d]^{U''}\\
(\E_2) & \C_2' \ar[r]_{i_2} & \C_2 \ar[r]_{F_2} &\C_2''.
}$$
Denote by $H_1$, $H_2$ the induced Hopf algebras of $F_1$ and $F_2$ respectively.
\begin{enumerate}
\item The morphism of exact sequences induces a Hopf
algebra morphism  $\phi : H_1 \to H_2$, and $U'$ is dominant (
respectively,  an equivalence) if and only if $\phi$ is a
monomorphism ( respectively,   and isomorphism);
\item If $U'$ and $U''$ are dominant, then so  is $U$.
\end{enumerate}
 \end{lemma}

\begin{proof}  1) We may assume $\C'_1 = \CoRep H_1$, $\C'_2 = \CoRep H_2$,  $U'$ being compatible with the forgetful functors. Then $U'$ is of the form $\phi_*$,
for some Hopf algebra morphism $\phi : H_1 \to H_2$,  $U'$ is an
equivalence if and only if $\phi$ is an isomorphism, and by
\cite[Remark  3.12]{tensor-exact}, $U'$ is dominant if and
only if $\phi$ is surjective.

2) Let $X$ be an object of $\C_2$. We are to show that $X$ is a
subobject of $U(Y)$ for some $Y \in \Ob(\C_1)$. Now  $F_2 U =
U'' F_1$ is dominant, so $F_2(X) \subset F_2 U(Z)$ for some $Z
\in \Ob(\C_1)$. Denote by $R_2$ the right adjoint of $F_2$; being
a right adjoint, it preserves monomorphisms so $R_2F_2(X) \subset
R_2 F_2 U(Z)$. We have $A_2= R_2(\un)$, and we have an isomorphism
$RF_2 \simeq A_2 \otimes ?$ coming from the fact that the
adjunction $(F_2,R)$ is a Hopf monoidal adjunction (see
\cite[Proof of Proposition 6.1]{tensor-exact}, and \cite{BLV} for
the original statement in terms of Hopf monads).

Thus $A_2 \otimes X \subset A_2 \otimes U(Z)$. Since we have $\un
\subset A_2$,  we obtain $X \subset A_2 \otimes U(Z)$. On the
other hand, $A_2$ belongs to the essential image of $ i_2$
and $U'$ is dominant, so there exists $Z' \in \Ob(\C'_1)$ such
that $A_2 \subset  i_2  U'(Z') \simeq U i_1 (Z')$.
Consequently $X \subset U( i_1 (Z') \otimes Z)$, which
shows that $U$ is dominant, as claimed.
\end{proof}

\subsection{Central exact sequences of finite tensor categories}

If $i: \C' \to \C$ is a strong monoidal functor between monoidal categories, a \emph{central lifting of $i$} is a
strong monoidal functor $\tilde{i}: \C' \to \Z(\C)$ such that $\U\,\tilde{i} = i$, where $\U$ denotes the forgetful functor $\Z(\C) \to \C$.
Note that if $i$ is full, given a central lifting $\tilde{i}$ there exists a unique braiding on $\C'$ such that $\tilde{i}$ is braided.

An exact sequence of finite tensor categories $\C'
\overset{i}\toto \C \overset{F}\toto \C''$, with induced central
algebra $\A=(A,\sigma)$,  is  called  \emph{central} if the
restriction of the forgetful functor $\U: \Z(\C) \to \C$ induces
an equivalence of categories  $\bra{\A} \to \bra{A}$.

\begin{theorem}
 Consider
 an exact sequence of finite tensor categories

 $$(\E) \quad \C' \overset{i}\toto \C \overset{F}\toto \C''$$ with induced central algebra $(A,\sigma)$ and induced central coalgebra $(\hat{C},\hat{\sigma})$.
The following assertions are equivalent:
\begin{enumerate}[(i)]
\item The exact sequence $(\E)$ is central;
\item There exists a central lifting $\tilde{i}$ of $i$ such that $\tilde{i}(A)= (A,\sigma)$;
\item There exists a central lifting $\tilde{i}$ of $i$ such that $\tilde{i}(\hat{C}) = (\hat{C},\hat{\sigma})$.
\end{enumerate}
Moreover, if these assertions hold, the central liftings $\tilde{i}$ of $i$ appearing in assertions (ii) and (iii) are essentially unique and they coincide.
\end{theorem}

The central lifting $\tilde{i}$ is called the \emph{canonical central lifting of the central exact sequence $(\E)$.}

\begin{proof}
Assertions (ii) and (iii) are equivalent because $(A,\sigma)$ is the right dual of $(\hat{C},\hat{\sigma})$ and strong monoidal functors preserve duals.

Denote by $j$ the tensor functor $\bra{(A,\sigma)} \to \bra{A}$ induced by the forgetful functor $\U$. In particular $j(\A)=A$.

We have (ii) $\implies$ (i) because if $\tilde{i}$ is a central
lifting of $i$ such that $\tilde{i}(A) = (A,\sigma)$, then
$\tilde{i}(\bra{A}) \subset \bra{(A,\sigma)}$ and by definition of
a central lifting, $j\tilde{i} = \id_{\bra{A}}$. This shows that
$j$ is full and essentially surjective; since on the other hand
$j$ is faithful, it is therefore an equivalence, with
quasi-inverse $ \tilde{i}$.

We have (i) $\implies$ (ii) because if $j$ is an equivalence, then it admits a quasi-inverse $k$, which is also a tensor functor.
One defines $k$ by picking for each object $X$ of $\bra{A}$ an object $k(X)$ in $\Z(\C)$ such that $\U k(X) \simeq X$. One may further impose that
$\U k(X)=X$, and $k(A) = (A,\sigma)$. Then $\tilde{i}: \bra{A} \to \Z(\C), X \mapsto k(X)$ is a central lifting of $i$ sending $A$ to $(A,\sigma)$,
which proves assertion (ii).

If $\tilde{i}$ exists, it is a quasi-inverse of $j$ and as such, it is essentially unique.
\end{proof}


\begin{example}\label{equiv-nec}
If $G$ is a finite group acting on a tensor category $\C$ by
tensor autoequivalences, then the  corresponding exact sequence of
tensor categories
$$\rep G \overset{i}\toto \C^G \overset{F}\toto \C$$ is central.
Indeed, if $(V,r)$ is a representation of $G$ and $(X,\rho)$ an
object of $\C^G$, one verifies that the trivial
isomorphism $V \otimes X \iso X \otimes  V$ lifts to an
isomorphism $\sigma_{(V,r),(X,\rho)} = i(V,r) \otimes (X,\rho)
\simeq (X,\rho) \otimes i(V,r)$ in $\C^G$, and this defines a
central lifting  $\tilde{i} : \rep G \to \Z(\C^G)$ of $i$, $(V,r)
\mapsto  ((V,r), \sigma_{(V,r), -})$. Moreover, the
induced central algebra $(A,\sigma)$  of $F$ is defined by
$A=\kk^G$, with $G$-action defined by  right translations,
and $\sigma=\sigma_{A, -}$, hence centrality of the exact
sequence.
\end{example}

If $\C' \to \C \to \C''$ is a central exact sequence of tensor categories, then $\C'$ is symmetric. More precisely, we have the following lemma.

\begin{proposition}\label{tannakian} Consider a central exact sequence
$$(\E) \qquad \C' \overset{i}\toto \C \overset{F}\toto \C'',$$
with canonical central  lifting $\tilde{i}$. Then the induced Hopf
algebra $H$ of $(\E)$ is commutative, so that $\C' \simeq \CoRep
H$ is endowed with a symmetry, and with this symmetry on $\C'$,
the tensor functor $\tilde{i}: \C' \to \Z(\C)$ is braided. If in
addition $H$ is split semisimple, then $G=\Spec H$ is a discrete
finite group and $\C \simeq \rep G$.
\end{proposition}

\begin{remark}\label{rem-tannakian} As a special case of Proposition~\ref{tannakian}, for any  finite dimensional Hopf algebra $H$ the corresponding exact sequence
$$\CoRep H \to \CoRep H \to \vect_\kk$$
is central if and only if $H$ is commutative.
\end{remark}

\begin{remark}
Note that if $\C' \to \C \to \C''$ is a central exact sequence
of fusion categories over an algebraically closed field of characteristic $0$, $G= \Spec H$ is a discrete finite group and we have $\C' \simeq \rep G$.
\end{remark}


\begin{proof}[Proof of \ref{tannakian}]   Let $(A,\sigma)$ be the induced central  algebra of $(\E)$.
We may replace $(\E)$ with the equivalent exact sequence $$\KER_F
= \bra{A} \toto \C \toto \C''.$$ Consider the morphism of exact
sequences of tensor categories:
 $$\xymatrix{
\bra{A} \ar[r] \ar[d]_{=} & \bra{A} \ar[r] \ar[d]_{\text{incl.}} & \bra{\un}\ar[d]^{\text{incl.}}\\
\bra{A} \ar[r] &\C \ar[r] &\C''
}$$
Denote by $(\E_0)$ the top exact sequence in this diagram. Its induced central algebra is $(A,\sigma_{\mid \bra{A}})$.
Moreover it is
central, with canonical lifting $\tilde{i}_0$ defined by $\tilde{i}_0(X)= (X,{s_X}_{\mid \bra{A}})$, where $\tilde{i}(X)=(X,s_X)$
denotes the canonical lifting of $(\E)$.

Thus, it is enough to prove the  theorem  for $(\E_0)$. We
may again replace $(\E_0)$ by the equivalent exact sequence
$$\CoRep  H \to \CoRep H \to \vect_\kk,$$
where $H$ is the induced Hopf algebra of $(\E)$, which is the situation of Remark~\ref{rem-tannakian}.

In that situation, we have $A=(H,\Delta)$ and  the half-braiding $\sigma$ is defined by
\begin{equation}\label{sigma-H}\begin{array}{lllll}
\sigma_{V} &: & H \otimes V & \to & V \otimes H \\
& & h \otimes v & \mapsto & v_{(0)} \otimes S(v_{(1)}) hv_{(2)} \\
\end{array}\end{equation}
in Sweedler's notation, for any finite dimensional right $H$-comodule $V$ (see \cite[Example 6.3]{tensor-exact}).

Centrality of the exact sequence means that we have a central lifting of the identity of $\CoRep H$,
which is nothing but a braiding $c$ on $\CoRep H$, and in addition this braiding is required to be such that $c_{A,?} = \sigma$.
What we have to prove is that $H$ is commutative and $c$ is the standard symmetry.

Let $r: H \otimes H \to \kk$ be the coquasitriangular structure
corresponding  to  this braiding, so that for any pair of
right comodules $V$ and $W$,  we have
\begin{equation}\label{r}c_{V, W}: V \otimes W \to W \otimes V,
\;  c_{V, W} (v \otimes w) = r(v_{(1)}, w_{(1)})w_{(0)}
\otimes v_{(0)}. \end{equation} Comparing \ref{sigma-H} and
\ref{r}, we obtain by a straightforward computation that the
condition $c_{A,A} = \sigma_A$ implies $r = \varepsilon \otimes
\varepsilon$. That means that the forgetful functor $\CoRep H \to
\vect_\kk$ is braided, so $H$ is commutative and $c$ is
the standard symmetry. This concludes the proof of the
proposition.
\end{proof}

\subsection{Normality and centrality criteria}\label{criterio}

The following theorem gives a sufficient condition for a dominant
tensor functor between  finite tensor  categories to be
normal in terms of the induced central algebra.

\begin{theorem}\label{sum-inv} Let $F: \C \to \D$ be a dominant tensor functor between
finite tensor categories $\C, \D$, and let $(A, \sigma)$ be its induced
central algebra. Assume that  $A$ decomposes as a direct
sum of invertible objects of $\C$. Then:
\begin{enumerate}[(i)]
\item The functor $F$ is normal,
the isomorphism classes of simple direct summands of $A$ form a group $\Gamma$, and we have an exact sequence of tensor categories
$$(\E) \qquad \Gamma\ti\vect  \toto \C \overset{F}\toto \D,$$
where $\Gamma\ti\vect$ denotes the tensor category of finite dimensional $\Gamma$-graded vector spaces.
\item If in addition $(A,\sigma)$ decomposes as a direct sum of invertible objects in $\Z(\C)$, then the exact sequence ($\E$) is central.
\end{enumerate}
\end{theorem}

\begin{proof} An invertible object in a tensor category is both simple and scalar.
Let $R$ denote the right adjoint of $F$. For any
invertible object $g$ of $\C$, we have by adjunction $\dim
\Hom_\C(g,A) =\dim \Hom_\D(F(g),\un)$,  because $A=R(\un)$.
 Now $F(g)$ is
invertible in $\D$, so $\dim \Hom_\C(g,A) = 1$ if $F(g) \simeq
\un$, and $\dim \Hom_\C(g,A) = 0$ otherwise.

In other words: (1) the invertible factors of $A$ are exactly the invertible objects of $\C$ which are trivialized by $F$; therefore, they form a group for the tensor product; and (2) their multiplicity in $A$ is exactly one.

In particular if $A$ is the direct sum of its invertible factors, then $A$ itself is trivialized by $F$, that is, $F$ is normal. In that case, we have an exact sequence $\KER_F \to \C \to \D$.
Now $\KER_F$ is the tensor subcategory of $\C$ generated by $A$; it is a pointed tensor category whose invertible objects are the invertible factors of $A$, whose isomorphism classes form a group $\Gamma$. So $\KER_F$ is a pointed tensor category; in addition, $\KER_F$ admits a fiber functor, hence it is tensor equivalent to the category $\Gamma\ti\vect$. This proves assertion (i).

Now assume that $(A,\sigma)$ decomposes as a direct sum of
invertible objects of $\Z(\C)$,  that is,  $(A,\sigma) =
\bigoplus_{i=1}^{n} (g_i,\sigma_i)$. Then $A = \bigoplus_{i=1}^{n}
g_i$, where the $g_i$'s are invertible in $A$, so that the first
part of the theorem applies. The category $\bra{(A,\sigma)}$ is
generated as a tensor category by the $(g_i,\sigma_i)$.  Let us
show that $\bra{(A,\sigma)}$ is \emph{additively} generated by the
$(g_i,\sigma_i)$. For this it is enough to show that
$(g_i,\sigma_i) \otimes (g_j,\sigma_j)$ is a direct factor of
$(A,\sigma)$ for all $i, j \in \{1, \dots n\}$. Now the product
$\mu: A \otimes A \to A$ embeds $g_i \otimes g_j$ into $A$, and
lifts to a morphism in $\Z(\C)$, namely the product of
$(A,\sigma)$; consequently it embeds $(g_i,\sigma_i)  \otimes
(g_j,\sigma_j)$ into $(A,\sigma)$. Thus $\bra{(A,\sigma)} \to
\bra{A}$ is full, which proves that $(\E)$ is central.
\end{proof}

\subsection{Example: Tambara-Yamagami categories}\label{ty}

 A Tambara-Yamagami category is a fusion category having exactly one non-invertible simple object $X$, with the additional condition that $X$ is not a factor of $X \otimes X$. These categories, which are in a sense the simplest non-pointed categories, have been classified in \cite{TY}.

Let $\TY$ be a Tambara-Yamagami category. Denote by $\Gamma$ the Picard group of $\TY$.  It is a finite abelian group.
Denote by $X$ a non-invertible simple object.
The maximal pointed fusion subcategory of $\C$, denoted by $\C_{\pt}$, is tensor equivalent to the category $\Gamma\ti\vect$ of
finite dimensional $\Gamma$\ti graded vector spaces.
The following proposition characterizes normal tensor
functors on $\C$.


\begin{proposition} Let $\TY$ be a Tambara-Yamagami category, with Picard group  $\Gamma$, and let $F: \C \to \D$ be a dominant tensor functor, with induced central algebra
$(A, \sigma)$. Then $F$ is normal if and only if $F$ is a fiber functor, or $A$ belongs to $\C_{\pt}$.

In the latter case, we have an exact sequence of tensor categories:
$$G\ti\vect \toto \TY \overset{ F}\toto \D,$$
where $G$ is a subgroup of $\Gamma$.
\end{proposition}


\begin{proof} The only proper fusion subcategories of $\C$ are those contained in
$\C_{\pt}$. This shows the `only if' direction.
Conversely, any fiber functor $F$ on $\C$ is normal, with
$\KER_F = \C$. Suppose on the other hand that $F$ is not a fiber functor.  Then $F$
is normal by Theorem \ref{sum-inv}. This finishes the proof of the
proposition.
\end{proof}

It is known that if a Tambara-Yamagami category $\TY$ admits a fiber functor, so that $\TY
\simeq \rep H$ for some semisimple Hopf algebra $H$,  then $H$
fits into an abelian exact sequence of Hopf algebras
$\kk^{\mathbb Z_2} \to H \to \kk\Gamma$ \cite{gp-ttic}. Hence in
this case $\TY$ fits into an exact sequence of fusion categories
$\rep \Gamma \to \TY \to \rep \mathbb Z_2$. In particular, $\TY$ is
not simple.

\subsection{Extensions of weakly integral fusion
categories}\label{extensions-wi}

Recall that a fusion category is weakly integral if its Frobenius-Perron dimension is a natural integer.
In this section we discuss dominant tensor functors on weakly
integral fusion categories.

\begin{lemma}\label{fpdim-A} Let $F: \C \to \D$ be a dominant tensor functor between
fusion categories, with induced central algebra $(A, \sigma)$.
Then we have:

\begin{enumerate} \item[(i)] $\FPind(\C: \D) = \FPdim A$.

\item[(ii)] If $X \in  \Irr(\C)$, then $m_X(A) \leq \FPdim
X$.

\item[(iii)] $F$ is normal if and only if for all $X \in  \Irr(\C)$ we have $m_X(A) = 0$ or $m_X(A) = \FPdim
X$. \end{enumerate}\end{lemma}

\begin{proof} By \cite[Proposition 4.3]{tensor-exact}, we have
$\FPind (\C: \D) = \FPdim R(\un)$. This proves (i) since $A =
R(\un)$.

By adjunction, we have $\Hom_\D(F(X), \un) \simeq \Hom_\C(X, A)$,
so that  $m_X(A) = m_\un(F(X))$. This implies (ii) since
$m_\un(F(X)) \leq \FPdim F(X) = \FPdim X$.

We next show (iii). The only if part follows from
\cite[Proposition 6.9]{tensor-exact}. Conversely, suppose that for
all $X \in  \Irr(\C)$ we have $m_\un(F(X)) = m_X(A) \in \{ 0,
\FPdim X\}$. Let $X \in \Irr(\C)$ and assume $m_\un(F(X)) \neq 0$.
Then  $m_\un(F(X)) = \FPdim X = \FPdim (F(X))$, so  $F(X)$ is
trivial. Thus $F$ is normal, which completes the proof of  (iii)
and of the lemma.
\end{proof}


\begin{proposition}\label{wint} Let $F: \C \to \D$ be a dominant tensor functor between
fusion categories $\C$ and $\D$ and let $(A, \sigma)$ be the  induced central
algebra of $F$. Then the following assertions are
equivalent:
\begin{enumerate}\item[(i)] $\C$ is weakly integral.
\item[(ii)] $\D$ is weakly integral and $\FPdim A \in \mathbb Z$.
\end{enumerate}
\end{proposition}

\begin{proof} (ii) $\Rightarrow$ (i) results immediately from Lemma \ref{fpdim-A} (i).

(i) $\Rightarrow$ (ii).
Notice first that since $\FPdim A \FPdim \D = \FPdim \C$ and
$\FPdim A$ is an algebraic integer, it is enough to verify that $\FPdim \D$ is a natural integer, that is, $\D$ is weakly integral.

Recall that $\D$ is tensor equivalent to the fusion category
$\C_A$ of right $A$-modules in $\C$. Since $A$ is an
indecomposable algebra in $\C$, the category ${}_A\C_A$ of
$A$-bimodules in $\C$ is a fusion category and it satisfies
$\FPdim {}_A\C_A = \FPdim \C$ \cite[Corollary 8.14]{ENO}.
Therefore ${}_A\C_A$ is weakly integral. We have a full tensor
embedding  $\C_A \subset  {}_A\C_A$.

Now, in a weakly integral fusion category the
Frobenius-Perron dimensions of simple objects are square roots of
natural integers \cite[Proposition 8.27]{ENO}, and as a result, a full
fusion subcategory of a weakly integral fusion category is weakly integral.
So $\D \simeq \C_A$ is weakly integral, and we are done.
\end{proof}

In the case where the functor $F$ is normal, we have $\FPdim A =
\FPdim \KER_F$, and since $\KER_F$ admits a fibre functor it is weakly integral. Thus we have:

\begin{corollary}\label{wi-closed} Let $\C' \to \C \to \C''$ be an exact sequence
of fusion categories. Then $\C$ is weakly integral if and only if
$\C''$ is weakly integral. In particular, the class of weakly integral fusion
categories is closed under extensions. \qed \end{corollary}

\section{Equivariantization revisited}\label{sect-equi-general}

The aim of this section is to state and discuss equivariantization
criteria. In order arrive at a synthetic statement, we have to
extend the notion of equivariantization to actions of finite group
schemes. Thanks to this generalization, we can state that an exact
sequence of finite tensor categories is central if and only if its
Hopf monad is normal cocommutative, and that it is an
equivariantization exact sequence if and only if it is perfect and
central, which extends a result of \cite{tensor-exact} concerning
discrete groups, and also reformulates and extends a result of
\cite{ENO2} concerning fusion categories.

\subsection{Cocommutative normal Hopf monads}
Let $\C$ be a tensor category. A \kt linear right exact normal Hopf monad $T$ on $\C$ is \emph{cocommutative} (see \cite{tensor-exact}) if
for any morphism $x: T(\un) \to \un$ and any object $X$ of $\C$
$$(x \otimes TX)T_2(\un,X) = (TX \otimes x)  T_2(X, \un).$$

Note that, if $V$ is a trivial object, and $X$ is an arbitraty object of $\C$, there is a canonical isomorphism $\tau_{V,X}: V \otimes X \iso X \otimes V$, which is characterized by the fact that for all $x: V \to \un$, we have $(X \otimes x) \tau_{V,X} = x \otimes X$. The natural isomorphism $\tau_{V,-}$ is a half-braiding,  called the \emph{trivial half-braiding of $V$.}

We have the following characterizations of normal cocommutative Hopf monads.

\begin{lemma}\label{def-cocom}
Let $\C$ be a tensor category and let $T$ be a normal Hopf monad
on $\C$, with induced central coalgebra $(\hat{C},\hat{\sigma})$.
The following assertions are equivalent:
\begin{enumerate}[(i)]
\item $T$ is cocommutative;
\item $T_2(X,\un) = \tau_{T\un,TX}T_2(\un,X)$ for $X$ in $\C$;
\item $\hat{\sigma}_{(M,r)} = \tau_{T\un,M}$ for $(M,r)$ in $\C^T$, or in short: $\hat{\sigma}$ `is the trivial half-braiding'.
\end{enumerate}
\end{lemma}

\begin{proof}
Assertion (ii) is just a reformulation of the definition of cocommutativity in terms of trivial half-braidings, so (i) $\iff$ (ii).
Let $(M,r)$ be a $T$\ti module. Since $T$ is a Hopf monad, we have fusion isomorphisms
\begin{eqnarray*}
\Phi^r_{(M,r)} &=(T\un \otimes r)T_2(\un,M):TM \iso T\un \otimes M, \\
\Phi^l_{(M,r)} &=(r \otimes T\un)T_2(M,\un): TM \iso M \otimes T\un,
\end{eqnarray*}
 and by definition $\hat{\sigma}= \Phi^l\Phi^{r -1}$. If (ii) holds, we have $\Phi^l_{(M,r)} = \tau_{T\un,M}\Phi^r_{(M,r)}$ by functoriality of $\tau$,
so $\hat{\sigma}_{(M,r)} = \tau_{T\un,M}$, which shows (ii) $\implies$ (iii). Conversely, applying (iii)   to $(M,r) = (TX,\mu_X)$ and composing on the right by $T(\eta_X)$ gives $T_2(X,\un) = \tau_{T\un,TX}T_2(\un,X)$, so (iii) $\implies$ (ii).
\end{proof}

From \cite[Theorem  5.21  and Theorem
5.24]{tensor-exact}, one deduces immediately
\begin{proposition}\label{equiv-ss} A dominant tensor functor  $F: \C \toto \D$  between finite tensor categories is an equivariantization under the action of a finite group $G$ if and only if the following two conditions are met:
\begin{enumerate}
\item the Hopf monad $T$ of $F$ is normal and  cocommutative;
\item the induced Hopf algebra $H$ of $F$ is split semisimple.
 \end{enumerate}
If these conditions hold, then $F$ is perfect, that is $T$ is exact, and $G =\Spec (H)$.
 \end{proposition}

This suggests that a (perfect) dominant tensor functor between
finite tensor categories is an equivariantization under the action
of a finite group scheme if and only if its Hopf monad is normal
cocommutative; the group scheme being the spectrum of the induced
Hopf algebra of $T$ - which, in this case, is a  finite
dimensional  commutative Hopf algebra.

We will now define group scheme actions in order to give a mathematical meaning to this claim, and then  prove it.

\subsection{Group scheme actions} \label{sect-group-scheme}
Let $\Aa$ be a monoidal category and let $\M$ be a category. An
action of $\Aa$ on $\M$ is a strong monoidal functor $\rho :\Aa
\to \EndFun(\M)$, where $\EndFun(\M)$ denotes the strict monoidal
category of endofunctors of $\M$. Given such an action $\rho$, we
say that $\M$ is an $\Aa$\ti module category, and we usually write
$\rho(a,m) = a \odot m$. Let $\M$, $\M'$ be two $\Aa$\ti module
categories. A functor of $\Aa$\ti module categories $\M \to \M'$
is a pair $(F,F_2)$, where $F :\M \to \M'$ is a functor
and $F_2$ is a natural isomorphism $F_2(a,m) : a \odot F(m) \iso
F(a \odot m)$, $a \in \Aa, m \in \M$, such that the following
diagrams commute:

$$
\xymatrix@!0 @R=1.1cm @C=1.7cm{
F((a\otimes b)\odot m) \ar[rr]^{\simeq}\ar[dd]_{F_2(a\otimes b,m)} && F(a \odot (b \odot m)) \ar[d]^{F_2(a,b\odot m)}\\
&& a\odot F(b\odot m)\ar[d]^{a \odot F_2(b,m)}\\
(a \otimes b) \odot F(m) \ar[rr]_{\simeq} && a\odot(b\odot F(m))
}\quad
\xymatrix@!0 @R=1.1cm @C=2cm{
F(\un \odot m) \ar[dr]^{\simeq} \ar[dd]_{F_2(\un,m)}& \\
& F(m)\\
\un \odot F(m) \ar[ur]_{\simeq}&}
$$
where the unlabeled isomorphisms come from the monoidal structure
of the  action.

Now assume $\Aa$ is endowed with a strong monoidal functor  $\omega : \Aa \to \vect_\kk$, $\M$ is a \kt category, and $\rho$ is an action of $\Aa$ on $\M$ by \kt linear endofunctors. The \emph{equivariantization of $\M$ under the action $\rho$} is the category $\M^\rho$, also denoted by $\M^\Aa$,  defined as follows. Objects of $\M^\rho$ are data $(m,\alpha)$ where $m$ is an object of $\M$ and $\alpha=(\alpha^\lambda_c)_{c \in \Ob(\C), \lambda \in \omega(c)^*}$ is a family of morphisms
$\alpha^\lambda_c : c \odot m \to m$ satisfying the following conditions:
\begin{enumerate}
\item \emph{functoriality :} $\alpha^\lambda_c$ is linear in $\lambda$, and if $f : c \to c'$ is a morphism in $\C$ and $\lambda \in \omega(c')^*$, then
$\alpha^\lambda_{c'} (f \odot m) = \alpha^{f^* \lambda}_c$, where $f^*\lambda = \lambda \omega(f)$;
\item \emph{$\rho$\ti  compatibility:} we have  commutative  diagrams
$$\xymatrix{
(a \otimes b)\odot m \ar[d]_\simeq \ar[r]^{\alpha^{\lambda \otimes \mu}_{a \otimes b}}& m&&\un \otimes m \ar[r]^{\alpha^{\omega_0}_\un} \ar[dr]_\simeq& m\\
a\odot(b \odot m)\ar[r]_{a\odot \alpha^\mu_b} & a \odot m \ar[u]_{\alpha^\lambda_a}&&& m \ar[u]_=
}$$with $a$, $b$ objects of $\Aa$ and $\lambda \in \omega(a)^*$, $\mu \in \omega(b)^*$.
\end{enumerate}
Morphisms in $\M^\Aa$ from $(m,\alpha)$ to $(n,\beta)$ are morphisms $f : m \to n$ in $\M$ satisfying $f\alpha = \beta(a \odot f)$.

Note that if $G$ is a discrete group, viewed as a monoidal category $\underline{G}$ whose objects are the elements of the group, and equipped with
the trivial strict monoidal functor $\omega : \underline{G} \to \kk$, $g \mapsto \kk$, then  a $\underline{G}$\ti action is the same thing as a  $G$\ti action in the usual sense, and in the case of a \kt linear action on a \kt category $\M$, $\M^{\underline{G}}$ is isomorphic to the usual equivariantization $\M^G$.

\begin{proposition}\label{action-normal}
Let $T$ be a \kt linear faithful exact normal Hopf monad on a
tensor category $\C$,  with induced Hopf algebra $H$. Let $
\Kh  = H^*$ and $\LL = \CoRep  \Kh$. Then there is a
natural action of $\LL$ on $\C$ by \kt linear endofunctors,
defined by
$$V \odot X = V  \square^\Kh  T(X),$$ and $\C^\LL$ is canonically isomorphic to
$\C^T$ as a \kt linear category.
\end{proposition}

\begin{proof}
For simplicity, we  identify  a finite dimensional vector
space $E$ with the trivial object $E\otimes \un$ in $\C$. Then
$T\un = H$ and the comonoidal structure of $T$ defines a structure
of $H$\ti bicomodule on $T(X)$. This enables us to define $V \odot
X = V  \square^\Kh  T(X)$ for $V$ a  finite-dimensional
right $ \Kh$\ti comodule and $X$ in $\C$. From the fact
that $T$ is a faithful exact Hopf monad, one deduces natural
isomorphisms $(V \otimes W) \odot X \simeq V \odot (W \odot X)$
and $(\kk,\varepsilon) \odot X \simeq X$ which make $\odot$ an
action of $\LL$ on $\C$. Let $A = (H,\Delta)$ be the
trivializing algebra of $\LL$. Then $A$ generates
$\LL$. An object $(m,\alpha)$ of $\C^\LL$ is entirely
determined by $\alpha^\varepsilon_A : A \odot X\simeq T(X) \to X$,
which can be interpreted as a $T$\ti action on $X$ because we have
a canonical isomorphism $A\odot X \simeq T(X)$. This defines a \kt
linear isomorphism $\C^\LL \iso \C^T$.
\end{proof}

Note that the action of Proposition~\ref{action-normal} is not compatible in any clear way with the tensor product of $\C$. In order to take care of the monoidal structure of $\C$,  we now introduce the notion of $\LL$\ti module tensor category.

Denote by $\ab_\kk$ the $2$-category of abelian \kt linear
categories having finite dimensional $\Hom$ spaces and objects of
finite length, $1$-morphisms being \kt linear left exact functors,
and $2$-morphisms being natural transformations. We equip
$\ab_\kk$ with a tensor product \emph{à la} Deligne, denoted by
$\boxtimes$, and characterized by the fact that given three
objects $\M$, $\M'$, $\M''$ in $\ab_\kk$, the category of \kt
linear left exact functors $\M \boxtimes \M' \to \M''$ is
equivalent to the category of functors $\M \times \M' \to \M''$
which are \kt linear left exact in each variable. This tensor
product makes $\ab_\kk$ a monoidal $2$\ti category with unit
object $\vect_\kk$, with a symmetry $\tau_{\M,\M'}: \M
\boxtimes \M' \iso \M' \boxtimes \M$ defined by $m \boxtimes m'
\mapsto m' \boxtimes m$.

Now if $\Aa$ is a tensor category over $\kk$, define an $\Aa$\ti module category to be an object $\M$ of $\ab_\kk$ endowed with
a \kt linear action of $\Aa$ such that the functor $\odot : \Aa \times \C \to \C$ is \kt linear right exact in the first variable (it is automatically exact in the second variable because $\Aa$ is autonomous). Thus we may view $\odot$ as a \kt linear right exact functor $\Aa \boxtimes \C \to \C$.


%

Let $L$ be a finite dimensional cocommutative Hopf algebra. The
tensor category $\LL = \CoRep L$ is tannakian; it is endowed with
a strong monoidal symmetric functor $\Delta_*: \LL \to \LL
\boxtimes \LL$, which is coassociative, and the symmetric fiber
functor  $\eps_*: \LL \to \vect_\kk$ is a counit for
$\Delta_*$. Thus $\LL$ is a bialgebra in the monoidal $2$-category
$\ab_\kk$.

If $(\M,\rho)$ and $(\M',\rho')$ are two $\LL$\ti module categories then one defines a new $\LL$\ti module category
$(\M, \rho) \boxtimes (\M',\rho') = (\M \boxtimes \M', \rho'')$, where
$$\rho'' = (\rho \boxtimes \rho')(\LL \boxtimes \tau_{\LL,\M} \boxtimes \N)
(\Delta_* \boxtimes \M \boxtimes \N)$$
and this tensor product defines a monoidal structure on the $2$\ti category of $\LL$\ti module categories.

 An  \emph{$\LL$\ti module tensor category} is a tensor
category over $\kk$, endowed with
\begin{itemize}
\item a structure of $\LL$\ti module $\rho: \LL \boxtimes \C \to \C$ of $\LL$ on $\C$;
\item natural isomorphisms
$$\alpha_{V,X,Y}: V \odot(X \otimes Y) \iso \otimes(V \odot (X \boxtimes Y))$$
$$\beta: V \boxtimes \un \iso V \otimes \un$$
making the tensor product $\otimes_\C: \C \boxtimes \C \to \C$ and the unit functor $u: \vect_\kk \to \C$, $\kk \mapsto \un$ morphisms of $\LL$\ti module categories.
\end{itemize}
If $\C$ is  an  $\LL$\ti module tensor category, then
$\C^\LL$ is monoidal.

\begin{proposition}
Let $\C$ be a tensor category over $\kk$.
\begin{enumerate}
\item Let $L$ be a finite dimensional cocommutative Hopf algebra, $\LL=\CoRep L$. Then a structure of $\LL$\ti module tensor category on $\C$ defines a \kt linear faithful exact Hopf monad $T = A \odot ?$ on $\C$, where $A=(H,\Delta)$;
\item If $T$ is a  normal cocommutative \kt linear faithful exact Hopf monad on $\C$, the action of $\LL = \CoRep L$ on $\C$ defined in Proposition~\ref{action-normal} makes $\C$ a $\LL$\ti module tensor category.
\end{enumerate}
Moreover, these construction are essentially mutually inverse. Given a Hopf monad as in Assertion (2) and the corresponding structure of $\LL$\ti module tensor category on $\C$, the canonical isomorphism  $\C^\LL \simeq \C^T$ is a tensor isomorphism.
\end{proposition}

\begin{proof}
Assume we have a structure of $\LL$\ti module tensor category on
$\C$, with action $\rho : \LL \to \EndFun(\C)$, and set $T=
\rho(A)= A \odot ?$. Then $T$ is a monad on $\C$, that is, an
algebra in $\EndFun(\C)$, because $A=(H,\Delta)$ is an algebra in
$\C$ and $\rho$ is strong monoidal. Moreover $T$ is \kt linear
exact, and it is faithful because $\rho$ is right exact and  $\un$
is a subobject of $A$, so $X = \un \odot X$ is a subobject of  $A
\odot X = T(X)$. Moreover the $\LL$\ti module tensor category
structure defines isomorphisms $T(X \otimes Y) \simeq T(X)
\square^\Kh  T(Y)$ and $T(\un) = A \odot \un \simeq  \Kh
 \otimes \un$, which  define a Hopf monad structure on $T$,
which is normal and clearly cocommutative.

Conversely, if $T$ is a normal cocommutative \kt linear faithful
exact  Hopf monad on $\C$, consider the action $\rho$ of $\LL$ on
$\C$ defined by $V \odot X = V  \square^\Kh  T(X)$. Then
$X \mapsto \rho(X)$ is \kt linear left exact, so the action may be
viewed as a \kt linear exact functor $\LL \boxtimes \C  \to \C$.
The structure of Hopf monad, cocommutativity and normality define
isomorphisms $A \odot (X \otimes Y) \simeq (A \odot X)
\square^\Kh  (A \odot Y)$ and $A\odot \un \simeq  L$,
which give rise to structures of morphisms of $\LL$\ti module
morphisms on the tensor product $\C \boxtimes \C \to \C$ and the
unit functor $\vect_\kk \to \C$, making $\rho$ a structure of
$\LL$\ti module tensor category on $\C$.
\end{proof}

%
%

%
%

%
%
%
%

\subsection{Equivariantization and centrality criteria}

Let $G$ be a finite group scheme over $\kk$. A \emph{tensor action
of $G$} on a  tensor category $\C$ is a structure of $\LL$\ti
module tensor category on $\C$, where $\LL = \CoRep \kk[G] =
\OO(G)\mbox{-}\!\textrm{mod}$, where  $\OO(G)$ is the Hopf
algebra of regular functions on $G$. The \emph{equivariantization
of $\C$} under a tensor action of a finite group scheme $G$  is
the tensor category $\C^G = \C^\LL$.

From the results of the previous section, we deduce:

\begin{theorem}\label{equiv-monad-general}
Let $F: \C \toto \D$ be a normal dominant tensor functor between finite tensor categories, and let $T$ be its Hopf monad and $H$ its induced Hopf algebra. The following assertions are equivalent:
\begin{enumerate}[(i)]
\item The tensor functor $F$ is an equivariantization under the tensor action of a finite group scheme on $\D$;
\item The normal Hopf monad $T$ is exact and cocommutative.
\end{enumerate}
If these assertions hold then  the induced Hopf algebra $H$ of $F$ is commutative and  the group scheme of assertion (i) is $G=\Spec H$.
\end{theorem}

\begin{theorem}\label{cocom-central}
Let $\C' \to \C \to \C''$ be an exact sequence of finite tensor categories, and let $T$ be the associated normal Hopf monad on $\C''$.
Then the following assertions are equivalent:
\begin{enumerate}[(i)]
\item The exact sequence $\C' \to \C \to \C''$ is central;
\item The normal Hopf monad $T$ is cocommutative.
\end{enumerate}
\end{theorem}

Theorem~\ref{cocom-central} will be proved in Section~\ref{proof-cocom-central}.

\begin{corollary}\label{dom-equiv} Consider a morphism of exact sequences of  finite  tensor categories
$$\xymatrix{
(\E_0)\ar[d]  &\C'_0 \ar[r]\ar[d]_W &\C_0 \ar[r]\ar[d]^U & \C''_0 \ar[d]^V\\
(\E) & \C' \ar[r] & \C \ar[r] &\C''
\
}$$
such that the vertical arrows are dominant tensor functors. Then
\begin{enumerate}\item
 if $(\E_0)$ is central, so is $(\E)$.
\item
If $(\E_0)$ is an equivariantization exact sequence for a finite
group scheme $G$, with  $G$  discrete  or $(\E)$ perfect,
then $(\E)$ is an equivariantization exact sequence for a subgroup
$G' \subset G$ acting on $\C''$ in a manner compatible with $V$.
Moreover if $W$ is an equivalence, then $G'= G$.
\end{enumerate}
\end{corollary}
\begin{proof}
Let $T_0$, $T$ be the normal Hopf monads, and $H_0$, $H$ the induced Hopf algebras of the exact sequences $(\E_0)$ and $(\E)$ respectively. Then we may assume that $(\E_0)$ and $(\E)$ are of the form  $\CoRep H_0 \to {\C''_0}^{T_0} \to \C''_0$ and $\CoRep H \to {\C''}^T \to \C''$ respectively, and we have a  diagram of tensor functors:
$$\xymatrix{
\CoRep H_0 \ar[r]\ar[d]_{W} &{\C''_0}^{T_0}\ar[r]^{U_{T_0}}\ar[d]^U & \C''_0 \ar[d]^{V}\\
\CoRep H \ar[r] & {\C''}^T \ar[r]_{U_T} &\C'' \ }$$ with $U'$,
$U$, $U''$ dominant, which commutes up to tensor isomorphisms.  By
transport of structure, we may assume that $U_T U = V U_{T_0}$ as
tensor functors.

Let us assume $(\E_0)$ is central, and let us show that $(\E)$ is
central. By  Theorem~\ref{cocom-central},
$T_0$ is cocommutative, and we are to prove that $T$ is cocommutative too.

If $(X,r)$ is an object of ${\C_0''}^{T_0}$ then $U(X,r) = (V(X),\lambda(X,r))$, so that we have a natural transformation
$\lambda:T VU_{T_0} \to V U_{T_0}$, which by adjunction can be encoded as a natural transformation $\Lambda: T V \to V T_0$ such that
$$U(X,r) = (VX, Vr \,\Lambda_X),\quad \mbox{for any $(X,r)$ in ${\C_0''}^{T_0}$.}$$
The transformation $\Lambda$ is compatible with the monad structures of $T_0$ and $T$, and it is comonoidal because $ U_{T} U =V U_{T_0}$ as tensor functors.

The tensor functor $W$ is induced by a morphism of Hopf algebras
$\phi: H_0 \to H$, which is surjective because $W$ is dominant
(see Lemma~\ref{exact-dominant}).  In particular $H$ is
commutative, and the group scheme $G = \Spec H$ is a subgroup of
the group scheme $G_0= \Spec (H_0)$ associated with the
central  exact sequence $(\E_0)$.

On the other hand, we have  $TV(\un)  = H^* \otimes \un$
and $V T_0(\un) = H_0^* \otimes \un$, and via these isomorphisms
$\Lambda_\un$ is the transpose of $\phi$; therefore
$\Lambda_\un$ is a monomorphism (in fact, one can show that
$\Lambda$ is monomorphism, a fact we will not use). Denote by
$(\hat{C}_0,\hat{\sigma}_0)$ and $(\hat{C},\hat{\sigma})$ the
induced central coalgebras of $T_0$ and $T$ respectively. Let
$(M,r)$ be a $T_0$\ti module. One deduces easily from the
comonoidality of  $\Lambda$  that the following diagram commutes:
$$\xymatrix{
T\un \otimes VM \ar[dd]_{\hat{\sigma}_{U(M,r)}}\ar[rr]^{\Lambda_\un \otimes VM}&& VT_0\un \otimes VM \ar[r]^{\simeq}& V(T_0\un \otimes M)
\ar[dd]^{V(({\hat{\sigma}_0})_{(M,r)})}\\
& TV(M) \ar[lu]^{\simeq}\ar[ld]_{\simeq}\ar[r]^{\Lambda_M}& VT_0(M) \ar[ur]_{\simeq}\ar[dr]^{\simeq}& \\
VM \otimes T\un \ar[rr]_{VM \otimes \Lambda_\un}&& VM \otimes VT_0\un \ar[r]_{\simeq}& V(M \otimes T_0\un)
}$$
where the slanted arrows are the fusion isomorphisms. Since $\Lambda_\un$ is a monomorphism and, by assumption, $\hat{\sigma}_0$ is the trivial half-braiding, we see that $\hat{\sigma}_{U(M,r)}$ is the trivial half-braiding. The tensor functor $U$ being dominant,  $\hat{\sigma}$ is the trivial half-braiding, that is, $T$ is cocommutative.

In particular if $(\E_0)$ is an equivariantization under the
action of a group scheme $G$,  it is central so $T$ is
cocommutative. If $G$ is  discrete  or $(\E)$ is perfect
(hence $T$ is exact), then by Proposition~\ref{equiv-ss} or
Theorem~\ref{equiv-monad-general} $(\E)$ is an equivariantization
exact sequence, corresponding with a tensor action of $G \subset
G_0$ on $\C''$, which by construction is compatible with the
tensor action of $G_0$ on $\C_0''$ via $V$. If $W$ is an
equivalence, $\phi$ is an isomorphism so $G= G_0$.
\end{proof}

\section{Equivariantization: special cases}\label{sect-special}

\subsection{The braided case}

A \emph{braided exact sequence of tensor categories} is an exact sequence of tensor categories
$$\C' \overset{i}\toto \C  \overset{F}\toto \C''$$
such that the tensor categories $\C$, $\C''$ and the tensor functor $F$ are braided. This implies that $\C'$ admits a unique braiding such that $i$ is braided, too.

\begin{proposition}\label{equiv-braided}
A braided exact sequence of finite tensor categories is  central.
In particular, if it is perfect,  it is an equivariantization
exact sequence.
\end{proposition}

\begin{proof} Let $\C' \overset{i}\toto \C  \overset{F}\toto \C''$ be a braided exact sequence of tensor categories.
Then the Hopf monad of $F$ is braided by \cite[Proposition
5.29]{tensor-exact}, and it is \kt linear, right exact, and
normal, so it is cocommutative by \cite[Proposition
5.30]{tensor-exact}. Therefore, the exact sequence is central by
Theorem~\ref{cocom-central}  and, if it is perfect,   it is
equivariantization exact sequence by Theorem
\ref{equiv-monad-general}.
\end{proof}

\subsection{The fusion case}

Recall that if $\B$ is a braided category with braiding $c$, and $\Aa \subset \B$ is a set of objects, or a full subcategory of $\B$, then the \emph{centralizer of $\Aa$ in $\B$}, denoted by $C_\B(\Aa)$, is the full monoidal subcategory of $\B$ of objects $b$ satisfying $c_{b,a}c_{a,b} = \id_{a \otimes b}$
for any object $a$ in $\Aa$. If $\B$ is a fusion braided category, then $C_\B(\Aa)$ is a full fusion subcategory of $\B$ (see \cite{mueger}).

\begin{proposition}\label{fusion-equiv} Let $\kk$ be an algebraically closed field of characteristic $0$. Consider a central exact sequence of fusion categories
$$(\E) \qquad\C' \toto \C \toto \C''$$
with canonical lifting
$\tilde{i}: \C' \to \Z(\C)$, and set $\Aa = \tilde{i}(\C')$.
Let $C_{\Z(\C)}(\Aa)$ denote the centralizer of $\Aa$ in $\Z(\C)$.
Then the following holds:
 \begin{enumerate}[(1)]
\item we have a braided exact sequence of fusion categories (in fact a modularization exact sequence)
$$(\E_0)\qquad \Aa  \toto C_{\Z(\C)}(\Aa) \toto  \Z(\C'');$$
\item we have a morphism $(\E_0) \to (\E)$ of exact sequences of fusion categories:
$$\xymatrix{
\Aa \ar[d]_\simeq \ar[r] & C_{\Z(\C)}(\Aa)\ar[d] \ar[r] & \Z(\C'')\ar[d]\\
\C'\ar[r] & {\C} \ar[r] & \C''
,}$$
where the vertical arrows are dominant forgetful functors;
\item there is a finite group $G$ acting on $\Z(\C'')$ and on $\C''$ by tensor autoequivalences in a compatible way,
in such a way that  $\E_0$ and $\E$ are the equivariantization exact sequences relative to these actions.
\end{enumerate}

\begin{remark}
This proposition says essentially the same thing as
\cite[Proposition 2.10]{ENO2}(i), with a different viewpoint.
Indeed \cite[Proposition 2.10]{ENO2}(i) asserts that if $\C$ is a
fusion category over $\mathbb C$ and $\Aa \subset \Z(\C)$ is a
full tannakian subcategory such that the forgetful functor $\U:
\Z(\C) \to \C$ induces an equivalence of $\Aa$  with  a
tensor subcategory of $\C$, then $\Aa$ encodes a
de-equivariantization of $\C$, that is, a tensor action of a
finite group $G$ (such that $\Aa \simeq \rep G$) on a category
$\D$ such that $\C \simeq \D^G$. This can be deduced from
Proposition~\ref{fusion-equiv}, as follows. Assume $\Aa \subset
\Z(\C)$ is as above. Since $\Aa$ is tannakian, it contains a
self-trivializing semisimple commutative algebra $\A=(A,\sigma)$
such that $\Hom(\un,\A) = \kk$. The forgetful functor $\U$ induces
by assumption a full tensor embedding $i: \Aa \to \C$. Moreover,
we have a tensor functor $F_A: \C \to \C_A$ and an exact sequence
of fusion categories
$$(\E) \qquad\Aa \overset{i}\toto \C \overset{F_A}\toto \C_A.$$
The inclusion $\Aa \subset \Z(\C)$ is a central lifting of $i$ which makes $(\E)$ a central exact sequence, hence the tensor functor $\C \to \C_A$ is an equivariantization.
\end{remark}

\end{proposition}

\begin{proof}
Notice first that
assertion (3) derives immediately from assertions (1) and (2) and previous results: if (1) and (2) hold, then by Proposition~\ref{equiv-braided} $(\E_0)$ is an equivariantization exact sequence for the action of a finite group $G$ on $\Z(\C'')$. By Corollary~\ref{dom-equiv},  $(\E)$ is also an equivariantization exact sequence for an action of the same group $G$.

So the whole point is to construct the exact sequence $(\E_0)$ of
assertion (1)  and the morphism of exact sequences of assertion
(2). This is based on the following lemma.  Let us say that an
object $X$ of a braided category $\B$ with braiding $c$ is
\emph{symmetric} if $c_{X, X}^2 = \id_{X \otimes X}$.

\begin{lemma}\label{dys-nondeg} Let $\C$ be a fusion
category, let $\A = (A,\sigma)$ be a semisimple,  symmetric,
 self-trivializing commutative algebra in $\Z(\C)$ such that
$\Hom(\un,A) = \kk$, and let $\Aa = \bra{\A}$ be the fusion
subcategory of $\Z(\C)$ generated by $\A$.
Then $$C_{\Z(\C)}(\Aa)_\A = \dys \Z(\C)_\A \simeq \Z(\C_\A),$$
where $\dys \Z(\C)_\A$ denotes the category of  dyslectic $\A$\ti
modules in $\Z(\C)$.
\end{lemma}

\begin{proof}
Recall that if $\B$ is a braided category, with braiding $c$, and
$A$ is a commutative algebra in $\B$, then a \emph{dyslectic
$A$\ti module} (\cite[Definition 2.1]{pareigis}) is a right
$A$-module $(M,r: M \otimes A \to M)$ in $\B$ satisfying $r c_{\A,
M}c_{M, \A} = r$. The category $\dys \B_A$ of dyslectic $A$\ti
modules is a full monoidal  subcategory of $\B_A$ and it is
braided with braiding induced by $c$.

In the situation of the lemma, $\Z(\C)_\A$ is a
fusion category (because $\A$ is semisimple
and $\Hom(\un,\A)= \kk$), and $\dys \Z(\C)_\A$ is a full fusion
category of $\Z(\C)_\A$.

On the other hand, the fact that $\A$ is  symmetric  means
that it belongs to $C_{\Z(\C)}(\Aa)$, and the category
$C_{\Z(\C)}(\Aa)_\A$ is also a full fusion subcategory of
$\Z(\C)_\A$. Moreover, we have $$C_{\Z(\C)}(\Aa)_\A \subset \dys
\Z(\C)_\A$$ because a $A$\ti module $(M,r)$ such that $M$ belongs
to the centralizer of $\Aa$ is  dyslectic.

Now it follows from  \cite[Corollary 4.5]{schauenburg} that there is a natural  equivalence of braided tensor categories $\dys \Z(\C)_\A \simeq
\Z(\C_\A)$ which is compatible with the forgetful functors to $\C$.
%

All that remains to do is  to  show that the full, replete
inclusion of $C_{\Z(\C)}(\Aa)_\A$ in $\dys  \Z(\C)_\A$ is an
equality, which we do by showing that those two fusion categories
have the same Frobenius-Perron dimension. Now $\C_A$ is a fusion
category and $\FPdim \C_A = \FPdim\C/\FPdim A$ by Lemma
\ref{fpdim-A}(i). Since $\dys \Z(\C)_\A \simeq \Z(\C_A)$, we have
$$\FPdim \dys \Z(\C)_\A = \FPdim \Z(\C_A) = \left(\frac{\FPdim\C}{\FPdim A}\right)^2$$
by \cite[Proposition 8.12]{ENO}. On the other hand, $$\FPdim
C_{\Z(\C)}(\Aa)= \frac{\FPdim \Z(\C)}{\FPdim \A}$$ by
\cite[Theorem 3.14]{DGNOI}, since $\Z(\C)$ is a nondegenerate
fusion category. We have $\FPdim \Z(\C) =\FPdim (\C)^2$ again by
\cite[Proposition 8.12]{ENO},  and  since $\A$ is
self-trivializing, $\FPdim(\Aa) = \FPdim(\A) = \FPdim A$, because
 the forgetful functor $ \Z(\C)  \to \C$ preserves
Frobenisus-Perron dimensions. So
$$\FPdim C_{\Z(\C)}(\A)_\A= \frac{\FPdim(\C)^2}{\FPdim(A)^2} = \FPdim \dys \Z(\C)_\A,$$
and we are done.
This finishes the proof of the lemma.
\end{proof}

Now we apply the lemma. Let $\tilde{F}=F_\A: \C_{\Z(\C)}(\A) \to C_{\Z(\C)}(\A)_\A \simeq \Z(\C_A)$ be the functor `free $\A$\ti module' $X \mapsto X \otimes \A$. Then $\tilde{F}$ is a braided dominant normal fusion functor because $\A$ is semisimple and self-trivializing, so we have a braided exact sequence
of fusion categories
$$(\E_0)\qquad \Aa \to \C_{\Z(\C)}(\A) \to \Z(\C_A),$$
which is an equivariantization exact sequence by
Proposition~\ref{equiv-braided}, for an action of a certain finite
group $G$ on $\Z(\C)$. All our constructions are compatible with
the forgetful functors, hence  we get  a morphism of exact
sequences of fusion categories $(\E_0) \to (\E)$:
$$\xymatrix{
\Aa \ar[d]_\simeq \ar[r] & C_{\Z(\C)}(\Aa)\ar[d] \ar[r] & \Z(\C'')\ar[d]\\
\C'\ar[r] & {\C} \ar[r] & \C''}$$
The vertical arrow on the right is an equivalence because $(\E)$ is central and $\C' \simeq \bra{A}$, and the left vertical arrow is dominant because
it is the forgetful functor of the center. The middle vertical arrow is therefore dominant by virtue of Lemma~\ref{exact-dominant}, so assertion (2) holds.
By Corollary~\ref{dom-equiv}, $(\E)$
is an equivariantization exact sequence, for an action of $G$ on $\C$ which is compatible with the action of $G$ on $\Z(\C)$ and the forgetful functor
$\Z(\C) \to \C$.
\end{proof}

\subsection{The abelian case}

Let $\kk$ be a field. We say that a finite abelian group $G$ \emph{has the Kummer property} (\emph{w.r.t.} $\kk$) if $\kk$ contains $e$ distincts $e$-th roots of $1$, where $e$ is the exponent of $G$. If such is the case, the group of characters $\hat{\Gamma}_\kk$ of $G$ is isomorphic to $G$.
If $\kk$ is algebraically closed of characteristic $0$, all finite abelian groups have the Kummer property.


\begin{proposition}\label{equiv-abelian}
Let $F: \C \to \D$ be a dominant tensor functor between finite tensor categories over field $\kk$,  and denote by $\A=(A,\sigma)$ its induced central algebra. The following
assertions are equivalent:
\begin{enumerate}[(i)]
\item The functor $F$ is an equivariantization associated with an action of a finite abelian group
 $G$  having the Kummer property;
\item The induced central algebra $\A$ of $F$ is a direct sum of invertible
objects of $\Z(\C)$, and the finite abelian group $\Gamma$ formed
by the isomorphism classes of  these invertible objets has the
Kummer property;
\end{enumerate}
If these equivalent assertion hold, the groups $G$ of assertion (i) and $\Gamma$ of assertion (ii) are in duality, that is $G = \hat{\Gamma}_\kk$.
\end{proposition}

\begin{proof}
(i) $\implies$ (ii). Assume that $F$ is an equivariantization under a finite abelian group $G$ having the Kummer property. We have a tensor action of $G$ on $\D$, and we may assume that $\C= \D^G$. The exact sequence
$$\rep G \overset{i}\toto  \C \toto \D$$ is central,
that is, $i$ admits a central lifting  $\tilde{i}: \rep G \to
\Z(\C)$ such that $\tilde{i}(A)=(A,\sigma) = \A$ (see
Example~\ref{equiv-nec}). Now, since $G$ is abelian and has the
Kummer property, $\Rep G \simeq  \Gamma\ti \vect$ is
pointed, so $A$ splits as a sum of invertible objects, and so does
$\A = \tilde{i}(A)$, so (ii) holds.

(ii) $\implies$ (i).  Assume that $\A$ splits as a direct sum of
invertible objects of $\Z(\C)$.  Then by Theorem~\ref{sum-inv},
$F$ is normal and fits into a central exact sequence
$$\C' \toto \C \toto \D.$$
Denote by $H$ the induced Hopf algebra of this exact sequence,
which is commutative by Proposition~\ref{tannakian}. The tensor
category $\C'$ is pointed with Picard group $\Gamma$ because it is
tensor equivalent to $\bra{\A}$ via the canonical lifting
$\tilde{i}$, and since $\C' = \CoRep H$, we see that $H$ is
cocommutative and split cosemisimple. Thus $G= \Spec H =
\hat{\Gamma}_\kk$ is a discrete abelian group, so $H$ is split
semisimple.  We conclude by Proposition~\ref{equiv-ss} that $F$ is
an equivariantization under the group $G$.
\end{proof}

\section{Equivariantization and the double of a Hopf monad}\label{proof-main}

\subsection{Relative centers and centralizers}

Let $\C$, $\D$ be monoidal categories, and let $F: \C \to \D$ be
 a  comonoidal functor. Define a half-braiding relative to
$F$ to be a pair $(d,\sigma)$, where $d$ is an object of $\D$ and
$\sigma$ is a natural transformation $d \otimes F \to F \otimes d$
satisfying:
\begin{eqnarray*}
&(F_2(c,c') \otimes d)\sigma_{c \otimes c'}&= (F(c) \otimes  \sigma_{c'})(\sigma_c \otimes F(c'))(d \otimes F_2(c,c')),\\
&(F_0 \otimes d)\sigma_{\un}& = \sigma_\un (d \otimes
F_0).
 \end{eqnarray*}
Half-braidings relative to $F$ form a category called the
\emph{center of $\D$ relative to $F$}  and denoted by $\Z_F(\D)$,
or $\Z_\C(\D)$ if the functor $F$ is clear from the context. It is
monoidal, with the tensor product defined by $$(d,\sigma) \otimes
(d',\sigma') = (d \otimes d',  (\sigma \otimes d')(d \otimes
\sigma')),$$ and the forgetful functor $\U: \Z_{F}(\D) \to \D$
is monoidal strict.

Now assume $F$ is strong monoidal (in particular, it can be viewed as a comonoidal functor).
Then we have a strong monoidal functor $\tilde{F}: \Z(\C) \to \Z_{F}(\D)$, defined by $\tilde{F}(c,\sigma) = (F(c),\tilde{\sigma})$, where $\tilde{\sigma} = F(\sigma)$ up to the structure isomorphisms of $F$.

If $F$ is strong monoidal and has a left adjoint $L$, then $T=FL$
is a bimonad on $\D$ and by adjunction,  $\Z_F(\D)$ is isomorphic
as a monoidal category to the center $\Z_T(\D)$ of $\D$ relative
to the bimonad $T$ defined in \cite[Section  5.5]{BV}.

Let $(\E)=\bigl(\C' \overset{i}\toto \C \overset{F}\toto \C''\bigr)$ be an exact sequence of finite tensor categories over a field $\kk$.
We will show that, if $(\E)$ is central, with canonical central lifting $\tilde{i}$, then we have an exact sequence of tensor categories
$$\C' \overset{\tilde{i}}\toto \Z(\C) \overset{\tilde{F}}\toto \Z_F(\C'')$$ and a morphism of exact sequences of tensor categories
$$\xymatrix{
\C' \ar[r] \ar[d]_{=} & \Z(\C)\ar[r]\ar[d] & \Z_F(\C'')\ar[d]\\
\C'\ar[r] & \C\ar[r] & \C''.
}$$

\subsection{Quantum double of a Hopf monad}

Let $T$ be a Hopf monad on a rigid monoidal category $\C$. We say that $T$ is \emph{centralizable} (\cite{BV}) if for all objects $X$ in $\C$ the coend
$$Z_T(X) = \int^{Y \in \C} \ldual{ TY} \otimes X \otimes Y$$
exists. In that case the assignment $X \to Z_T(X)$ defines a Hopf
monad $Z_T$ on $\C$,  called the \emph{centralizer} of $T$.
Denoting by $j_{X,Y}: \ldual{TY} \otimes X \otimes Y \to Z_T(X)$
the universal dinatural transformation (in $Y$) associated with
the coend $Z_T(X)$, set $$\partial_{X,Y} = (TX \otimes j_{X,Y})
(\lcoev_{ TY} \otimes X \otimes Y): X \otimes Y \to TY
\otimes Z_T(X).$$ If $T$ is centralizable, we have an isomorphism
of tensor categories $\Z_T(\C) \iso \C^{Z_T}$, and so, an
isomorphism of tensor categories $K : \Z_{U_T}(\C) \iso \C^{Z_T}$.

Moreover, if $T$ is centralizable there also exists a canonical comonoidal distributive law $\Omega: T Z_T  \to Z_T T$, which is an isomorphism.
It is characterized by the following equation
\begin{equation}\label{omega}(\mu_X \otimes \Omega_Y)T_2(TY,Z_T(X))T(\partial_{X,Y}) =(\mu_X \otimes Z_T T(Y)) \partial_{TX,TY} T_2(X,Y).
\end{equation}
This invertible distributive law serves two  purposes:  it
defines (via its inverse) a lift $\tilde{T}$ of the Hopf monad $T$
to $\C^{Z_T}$, and it also defines a structure of a Hopf monad
$D_T = Z_T \circ_\Omega T$ on the endofunctor $Z_T T$ of $\C$. The
Hopf monad $D_T$ is called the double of $T$; it is
quasi-triangular, so that $\C^{D_T}$ is braided, and we have a
canonical braided isomorphism $K': \C^{D_T} \iso \Z(\C^T)$.

Lastly, we have a commutative diagram of tensor functors
$$\xymatrix{
\Z(\C^T) \ar[d]^{\tilde{U}_T} \ar[r]^{K'}_{\sim} &\C^{D_T} \ar[r]^{U_{D_T}} \ar[d]^{\tilde{U}}& \C^T \ar[d]^{U_T}\\
 \Z_{\tilde{U}_T}(\C)\ar[r]^{\sim}_{K} &\C^{Z_T}\ar[r]_{U_{Z_T}}& \C
}$$

If $\C$ is a finite tensor category over a field  $\kk$, and $T$ a
\kt linear right exact  Hopf monad on $\C$, then $\C^T$ is a
finite tensor category and the forgetful functor $U_T : \C^T \to
\C$ is a tensor functor. Moreover, $T$ is centralizable and $Z_T$
is a \kt linear  Hopf monad on $\C$, which is  right exact (being
an inductive limit of right exact functors) so $\Z_T(\C)\simeq
\C^{Z_T}$ is a finite tensor category and the forgetful functor
$\Z_T(\C) \to \C$ is a tensor functor.

%

\subsection{Proof of Theorem~\ref{cocom-central}}\label{proof-cocom-central}

Let $(\E)$ be an exact sequence of finite tensor categories over a field $\kk$.
Up to equivalence, we may assume that $(\E)$ is of the form
$$\bra{A} \toto \C^T \overset{U_T}\toto \C,$$
where $\C$ is a finite tensor category, $T$ is a \kt linear right exact normal Hopf monad on $\C$, and $A$ is the induced algebra of $U_T$.

We are to show that the following assertions are
equivalent:
\begin{enumerate}[(i)]
\item $T$ is cocommutative;
\item $(\E)$ is a central exact sequence.
\end{enumerate}
 Our  proof will rely on the following
\begin{lemma}\label{central-induced}
Let $T$ be a centralizable Hopf monad on an rigid category $\C$. Then
 the induced central algebra  (resp. coalgebra) of $U_T$ is the induced algebra (resp. coalgebra) of $\tilde{U}_T$.
\end{lemma}

Before we prove this lemma, let us show how it enables us to
conclude. Consider the tensor functor $\tilde{U}_T: \Z(\C) \to
\Z_T(\C)$. It is dominant.  Indeed $U_T$ is dominant by
assumption,   which means that the unit of $T$ is a
monomorphism, and so is the unit of $\tilde{T}$ because it is a
lift of $T$.

Moreover, $\tilde{U}_T$ is normal if and only if $T$ is cocommutative. This can be seen as follows.
Denote by $(\hat{C}, \hat{\sigma})$ the induced central coalgebra of $U_T : \C_T \to \C$, which is also the  induced coalgebra of $\tilde{U}_T$ by Lemma~\ref{central-induced}. We have $\tilde{T}(\un) = \tilde{U}_T(\hat{C},\hat{\sigma})$, and also $T(\un)= \hat{C}$.
In particular, $\tilde{U}_T$ is normal if and only if $(\hat{C},\hat{\sigma})$ is trivial in $\Z_T(\C)$, that is $\hat{C}$ is trivial in $\C$ (which is true because we have assumed  $T$ is normal) and $\hat{\sigma}$ coincides with the trivial half-braiding. The latter condition means that $T$ is cocommutative by Lemma~\ref{def-cocom}.

Denote by $\A=(A,\sigma)$ the induced central algebra of $U_T$, which is the right dual of $(\hat{C},\hat{\sigma})$. It is also the induced algebra of $\tilde{U}_T$.

Now assume $T$ is cocommutative. As we have just seen this means that $\tilde{U}_T$ is normal and dominant, so we have an exact sequence of tensor categories
$$(\E_0)\qquad \bra{\A} \toto \Z(\C) \overset{\tilde{U}_T}\toto \Z_{T}(\C).$$
Moreover,  we  have a morphism of exact sequences of
tensor categories
$$\xymatrix{
\bra{\A} \ar[r] \ar[d]_{V} & \Z(\C^T) \ar[r] \ar[d]_U & \Z_T(\C)\ar[d]_{W}\\
\bra{A} \ar[r] & \C^T \ar[r] & \C
}
$$
where $U$, $V$, $W$ denote the forgetful functors. We have $U(\A) = A$, so $V$ is an equivalence of categories, that is, $(\E)$ is central.

Conversely, assume $(\E)$ is central. That means that the forgetful functor induces a tensor equivalence $\bra{\A} \to \bra{A}$. Since $A$ is self-trivializing, so is $\A$. But by Lemma~\ref{central-induced}, $\A$ is also the induced algebra of $\tilde{U}_T$, so this tensor functor is normal by Lemma~\ref{normal-algebra}; and as we have seen above,
this implies that $T$ is cocommutative. Thus, we have shown the equivalence of  (i) and (ii).
\begin{proof}[Proof of Lemma~\ref{central-induced}]
The induced  (central) algebra being the dual of the  induced (central) coalgebra, it is enough to prove the assertion for coalgebras.
Let $\tilde{L}$ denote the left adjoint of $\tilde{U}$. The induced coalgebra $\tilde{C}$ of $\tilde{U}_T$ is $K^{-1}\tilde{L}(\un)$.

The functor $\tilde{U}$ is monadic, and its monad $\tilde{T}$ is the lift of $T$ defined by the distributive law $\Omega^{-1}$. This means that we have
$\tilde{L}(c,r)= (T(c),T(r) \Omega^{-1}_c)$ for $(c,r)$ in $\C^{Z_T}$.
We also have $K^{-1}(c,\rho) = ((c,r),s)$, where $r=\rho {\eta_{Z_T}}_{Tc}$ and $s$ is the half-braiding defined by
$s_{(x,r)}=(r \otimes \rho Z_T(\eta_c))\partial_{c,x}$ for $(x,r)$ in $\C^T$.
As a result,
$\tilde{C}= ((T\un,\mu_\un),\Sigma)$, where $$\Sigma_{(c,r)}= (r \otimes T((Z_T)_0) \Omega^{-1}_\un)\partial_{T\un,c}.$$
On the other hand, we have $\hat{C}=(T(\un),\mu_\un)$, and the half-braiding $\tilde{\sigma}$ is characterized by the fact that the following diagram commutes:
$$\xymatrix@!0 @R=1.4cm @C=2.8cm{
T\un \otimes c \ar[rr]^{\hat{\sigma}_{(c,r)}}& & c \otimes T\un\\
T\un \otimes Tc \ar[u]^{T\un \otimes r} &Tc \ar[r]_{T_2(c,\un)}\ar[ur]^\simeq \ar[l]^{T_2(\un,c)}\ar[ul]_\simeq& Tc \otimes T\un \ar[u]_{r \otimes T\un} }$$
so it is enough to verify: $\Sigma_{(c,r)}(T\un \otimes
r)T_2(\un,c)=(r\otimes T\un)T_2(\un,c)$. Now we have
\begin{align*}
\Sigma_{(c,r)}&(T\un \otimes r)T_2(\un,c)=(r \otimes T((Z_T)_0) \Omega^{-1}_\un)\partial_{T\un,c}(T\un \otimes r)T_2(\un,c)\\
&=(rT(r) \otimes T((Z_T)_0) \Omega^{-1}_\un)\partial_{T\un,Tc}T_2(\un,c)\quad\mbox{by functoriality of $\partial$}\\
&=(r\mu_c \otimes T((Z_T)_0) \Omega^{-1}_\un)\partial_{T\un,Tc}T_2(\un,c)\\
&=(r\mu_c \otimes T((Z_T)_0))T_2(Tc,Z_T\un)T(\partial_{\un,c})\quad\mbox{by \eqref{omega}}\\
&=(r T(r) \otimes T((Z_T)_0))T_2(Tc,Z_T\un)T(\partial_{\un,c})\\
&=(r \otimes T\un)T_2(c,\un)T((r \otimes (Z_T)_0)\partial_{\un,c})\quad\mbox{by functoriality of $T_2$}\\
&=(r \otimes T\un)T_2(c,\un)T(r\eta_c)\quad\mbox{by construction of $(Z_T)_0$}\\
&=(r \otimes T\un)T_2(c,\un).
\end{align*}
This concludes the proof of the lemma.\end{proof}

\section{Tensor functors of small Frobenius-Perron index}\label{demo}

In this section we prove Theorems \ref{index-2} and \ref{index-p}.

\subsection{Tensor functors of Frobenius-Perron index 2}

\begin{theorem}\label{index-2} Let $F: \C \to \D$ be a dominant tensor functor
between fusion categories over a field of characteristic $0$. If
$\FPind(\C: \D) = 2$, then $F$ is an equivariantization associated
with an action of $\Zed_2$ on $\D$.
\end{theorem}

\begin{proof}
Let $F: \C \to \D$ be a dominant tensor functor of Frobenius-Perron index $2$ between fusion categories.
By \cite[Proposition 4.13]{tensor-exact}, $F$ is normal and so
it induces an exact sequence of fusion categories
\begin{equation}\label{z2}\rep \mathbb \Zed_2 \to
\C \overset{F}\to \D.\end{equation}

Now let $\A=(A, \sigma)$ be the induced central algebra of $F$. We
have $\FPdim \A = \FPdim A = 2$ by Lemma~\ref{fpdim-A}. Since $\A$
contains the unit object, we have $\A = \un \oplus S$, with $S$
invertible. By  Proposition~\ref{equiv-abelian}, $F$ is an
equivariantization relative to an action of $\Zed_2$. This
concludes the proof of the theorem.
\end{proof}

\subsection{Tensor functors of small prime Frobenius-Perron index}\label{index-p}

\begin{theorem}\label{index-p} Let $F: \C \to \D$ be a dominant tensor functor
between fusion categories over a field of characteristic $0$.
Assume that $\FPdim\, \C$ is a natural integer, and that
$\FPind(\C: \D)$ is the smallest prime number dividing
$\FPdim\,\C$. Then $F$ is an equivariantization associated with an
action of $\Zed_p$ on $\D$.
\end{theorem}

\begin{proof}
 Assume $\C$ is a weakly integral fusion category and
let $p$ be the smallest prime factor of $\FPdim \C$. Consider a
dominant tensor functor  $F: \C \to \D$, where $\D$ is a fusion
category, such that $\FPind(\C: \D) = p$.

Recall that a fusion category is \emph{integral} if its objects all have integral Frobenius-Perron dimension.
If $\C$ is not integral, then by \cite[Theorem 3.10]{gel-nik} it is $\Zed_2$-graded, and so $\FPdim \C$ is even.  Thus  $p = 2$,  and so Theorem~\ref{index-2} applies and we are done.

From now on we assume that $\C$ is integral. Then $\Z(\C)$ is also an integral fusion category. Let $\A = (A, \sigma)$ be the induced central algebra of $F$. We have $\FPdim \A = \FPdim A = \FPind(\C: \D) = p$. Let us
decompose $\A$ as a direct sum of simple objects of $\Z(\C)$:
\begin{equation}\A = W_1 \oplus \dots \oplus W_r. \end{equation}
We have $r \ge 2$ because $\A$ is not simple (it contains the unit object), so the Frobenius-Perron dimension of $W_i$ is an integer
$< p$  for all $i$.

The center $\Z(\C)$  is a non-degenerate braided fusion category. According to \cite[Theorem 2.11]{ENO2} (i),
$(\FPdim W_i)^2$ divides $\FPdim \Z(\C) = (\FPdim \C)^2$, and so $\FPdim W_i$ divides $\FPdim\C$.
We have $\FPdim W_i = 1$, because  $p$ is by assumption the smallest prime divisor of $\FPdim \C$, and so $W_i$ is invertible in $\Z(\C)$.

This implies that $\A$ belongs to $\Z(\C)_{\pt}$. By Theorem~\ref{sum-inv},
we have an exact sequence
\begin{equation}\label{zp}\rep \mathbb Z_p \to
\C \overset{F}\to \D.\end{equation} which is central, and by Theorem~\ref{equiv-monad-general},
it is an equivariantization exact sequence.
This concludes the proof of the theorem.
\end{proof}

\subsection{Fusion subcategories of index $2$ are not always normal}
Let $\C$ be a fusion category.
A full fusion  subcategory $\D \subset \C$ is \emph{normal in $\C$} (see \cite{tensor-exact}) if the inclusion $\D \subset \C$  extends to an exact sequence of fusion categories
$\D \to \C \to \C''$.
In that case we have $\FPdim \C'' = \frac{\FPdim \C}{\FPdim \D}$.

If $\D$ is a full fusion subcategory of $\C$ then the ratio $\frac{\FPdim \C}{\FPdim \D}$ is an algebraic integer (see \cite[Proposition 8.15]{ENO}). If $\C$ is weakly integral, so is $\D$, and therefore
$\frac{\FPdim \C}{\FPdim \D}$ is a natural integer.

Is it true that if $\C$ is weakly integral and $\frac{\FPdim \C}{\FPdim \D}$ is the smallest prime number dividing $\FPdim \C$, then $\D$ is normal in $\C$?
We show that even for $p=2$ such is not the case, by exhibiting counterexamples in Tambara-Yamagami categories (see Section~\ref{ty}).

\begin{proposition}
Let $\C$ be a Tambara-Yamagami category. Then we have
$$\frac{\FPdim \C}{\FPdim \C_{\pt}} = 2,$$
but $\C_{\pt}$ is not normal in  $\C$ if the order of the Picard group of $\C$ is not a square.
\end{proposition}

\begin{proof}
Let  $\Gamma = \Pic(\C)$ and denote by $X$ the simple non-invertible object of $\C$.
We have $X \otimes X \simeq \sum_{g \in \Gamma} g$, so $\FPdim X = \sqrt{|\Gamma]}$.
We have  $\C_{\pt} = \bra{\Gamma}$, and
$$\frac{\FPdim \C}{\FPdim \C_\pt} = \frac{2 |\Gamma|}{|\Gamma|}=2.$$  Now assume assume that $|\Gamma|$ is not a square. Then $\C$ is not integral because $\FPdim X$ is not an integer.
We conclude by the following lemma.

\begin{lemma}\label{lem-index-integral} Let $\C$ be a fusion category and let $\D \subset \C$ be
a normal fusion subcategory such that
 $\frac{\FPdim \C}{\FPdim \D}$ is prime. Then  $\C$ is integral.
\end{lemma}

\begin{proof} Consider the exact sequence of fusion categories $\D \toto \C \overset{F}\toto \C''$
coming from the fact that $\D$ is normal in $\C$. We have $\FPdim
\C'' = p$. By \cite[Corollary 8.30]{ENO}, $\C''$ admits a
quasi-fiber functor $\omega: \C'' \to \vect_\kk$. Then
$\C$ admits a quasi-fiber functor $\omega K$, and therefore  $\C$
is integral.
\end{proof}

A contrario the lemma shows that $\C_\pt$ is not normal in $\C$.
\end{proof}

\begin{proposition}\label{simple-ty}
Let $\C$ be a Tambara-Yamagami category with Picard group of prime order. Then $\C$ is a simple fusion category.
\end{proposition}

\begin{proof}
Assume we have an exact sequence of fusion categories $\D \to \C \to \C''$. Since $\C$  contains a simple object of Frobenius-Perron dimension $\sqrt{p}$, it is not integral. Consequently $\C''$ admit no quasi-fiber functor. In particular $\FPdim(\C'')$ is neither $1$ nor a prime number.
Therefore $\FPdim(\C'')=2p$, and $\D$ is trivial. Hence $\C$ is simple.
\end{proof}

\bibliographystyle{amsalpha}

\end{document}